\documentclass{amsart}

\usepackage{amsmath}
\usepackage{amsfonts}
\usepackage{amsthm} 
\usepackage{amssymb}

\usepackage{graphicx} 

\usepackage{color} 

\newtheorem{thm}{Theorem}[section]
\newtheorem{prop}[thm]{Proposition}
\newtheorem{cor}[thm]{Corollary}
\newtheorem{lem}[thm]{Lemma}
\newtheorem{rem}{Remark}
\newtheorem{exmp}{Example}[section]
\newtheorem{defn}[thm]{Definition}

\numberwithin{equation}{section} 

\makeatletter 
\@namedef{subjclassname@2010}{2010 Mathematics Subject Classification} 
\makeatother 

\begin{document}

\title[Generalized Fresnel integrals as oscillatory integrals]
{Generalized Fresnel integrals\\as oscillatory integrals\\with positive real power phase functions\\and applications to asymptotic expansions} 

\author{Toshio NAGANO and Naoya MIYAZAKI} 

\address{Department of Liberal Arts, Faculty of Science and Technology, Tokyo University of Science, 2641, Yamazaki, Noda, Chiba 278-8510, JAPAN} 
\email{tonagan@rs.tus.ac.jp}

\address{Department of Mathematics, Faculty of Economics, Keio University, Yokohama, 223-8521, JAPAN} 
\email{miyazaki@a6.keio.jp}

\thanks{The first author was supported by Tokyo University of Science Graduate School doctoral program scholarship 
and an exemption of the cost of equipment from 2016 to 2018 
and would like to thank to Professor Minoru Ito for 
giving me an opportunity of studies and preparing the environment.} 



\date {September 12, 2021} 

\begin{abstract} 
In this paper, we first generalize the Fresnel integrals by changing of a path 
for integration in the proof of the Fresnel integrals by Cauchy's integral theorem. 
Next, according to oscillatory integral, we also obtain further generalization of 
the extended Fresnel integrals.  
Moreover by using this result,  
we have an asymptotic expansion of an oscillatory integral with a positive real parameter, 
for a phase function with a degenerate critical point expressed by positive real power, including a moderate oscillation, 
and for a suitable amplitude function. 
This result gives a finer extension of the stationary phase method in one variable, 
which is known as a method for an asymptotic expansion of an oscillatory integral 
of a phase function with a non-degenerate critical point. 
\end{abstract} 

\maketitle 

\section{Introduction} 

In the present paper, 
we study a generalization of the Fresnel integrals: 
\begin{align} 
\int_{0}^{\infty} e^{\pm ix^{2}} dx 
= \frac{\sqrt{\pi}}{2} e^{\pm i \frac{\pi}{4}} \notag 
\end{align} 
and their applications to an asymptotic expansion of oscillatory integrals in one variable. 
In particular, we are interested in oscillatory integrals for phase functions with a degenerate critical point expressed by positive power, including a moderate oscillation, and for suitable amplitudes. 

As to proofs of the Fresnel integrals, 
several ways are known, for example \cite{Sugiura} I p.326, II p.85, 245, etc. 
In the proofs, we especially focus on the way of applying Cauchy's integral theorem 
to a holomorphic function $e^{-iz^{2}}$ on the domain with a fan of the center at the origin of Gaussian plane as a boundary(\cite{Ito-Komatsu} p.23). 
By changing the fan used in the proof with a holomorphic function $e^{-iz^{p}}z^{q-1}$ as an integrand, 
we can generalize the Fresnel integrals 
for $p>q>0$ in the following way: 
\begin{align} 
I_{p,q}^{\pm} 
:= \int_{0}^{\infty} e^{\pm ix^{p}} x^{q-1} dx 
= p^{-1} e^{\pm i\frac{\pi}{2} \frac{q}{p}} 
\varGamma \left( \frac{q}{p} \right), \notag 
\end{align} 
where $\varGamma$ is the Gamma function and double signs $\pm$ are in same order (Lemma \ref{Generalized the Fresnel integrals}). 
As to the case of $p > 0$ and $q > 0$, 
by making a sense 
of these integrals via oscillatory integral, 
we obtain 
\begin{align} 
\tilde{I}_{p,q}^{\pm} 
:= \lim_{\varepsilon \to +0} \int_{0}^{\infty} e^{\pm ix^{p}} x^{q-1} \chi (\varepsilon x) dx 
= p^{-1} e^{\pm i\frac{\pi}{2} \frac{q}{p}} 
\varGamma \left( \frac{q}{p} \right), \notag 
\end{align} 
where $\chi \in \mathcal{S}(\mathbb{R})$ with $\chi(0) = 1$ and $0 < \varepsilon < 1$ (Theorem \ref{th01} (i)). 
Moreover this is extended to a meromorphic function on $\mathbb{C}$ by analytic continuation (Theorem \ref{th01} (ii)). 
We call $\tilde{I}^{\pm}_{p,q}$ ``generalized Fresnel integrals''. 
This result can be considered as an extension of 
the case of $\lambda = q-1$ and $\xi = 1$ 
in the Fourier transform of Gel'fand-Shilov 
generalized function $\mathcal{F}[{x_{+}}^{\lambda}](\xi)$ 
with $\lambda \in \mathbb{C} \setminus \{ -1 \}$: 
\begin{align} 
\mathcal{F}[{x_{+}}^{q-1}](1) 
:= \lim_{\tau \to +0} \int_{0}^{\infty} e^{ix} x^{q-1} e^{-\tau x} dx 
= e^{i\frac{\pi}{2}q} \varGamma (q), \notag 
\end{align} 
where $\mathrm{Re} \tau > 0$ (\cite{Gel'fand-Shilov} p.170.). 

By using our result for $p>0$ and $q>0$, 
we can obtain an asymptotic expansion of an oscillatory integral with a positive real parameter, 
for a phase function with a {\bf degenerate critical point} expressed by positive real power, including a moderate oscillation, and for an amplitude function belonging to the class $\mathcal{A}^{\tau}_{\delta}(\mathbb{R})$\footnote{which is wider than the Schwartz space $\mathcal{S}(\mathbb{R})$,}(Definition \ref{A_tau_delta}) in the following way: for any $N \geq p+1$, as $\lambda \to \infty$, 
\begin{align} 
\lim_{\varepsilon \to +0} \int_{0}^{\infty} e^{\pm i \lambda x^{p}} a (x) \chi (\varepsilon x) dx 
= \sum_{k=0}^{N-[p]-1} \tilde{I}_{p,k+1}^{\pm} \frac{a^{(k)}(0)}{k!} \lambda ^{-\frac{k+1}{p}} + O\left( \lambda ^{-\frac{N-p+1}{p}} \right) 
\notag 
\end{align} 
(Theorem \ref{th02} (i)). 
In particular, if $p=m \in \mathbb{N}$, then for any $N > m$, as $\lambda \to \infty$, 
\begin{align} 
\lim_{\varepsilon \to +0} \int_{-\infty}^{\infty} e^{\pm i \lambda x^{m}} a (x) \chi (\varepsilon x) dx 
= \sum_{k=0}^{N-m-1} \tilde{c}_{k}^{\pm} \frac{a^{(k)}(0)}{k!} \lambda ^{-\frac{k+1}{m}} + O\left( \lambda ^{-\frac{N-m+1}{m}} \right), 
\notag 
\end{align} 
where $\tilde{c}_{k}^{\pm} := \tilde{I}_{m,k+1}^{+} + (-1)^{k} \tilde{I}_{m,k+1}^{\pm \pm_{m}}$ and 
\begin{align} 
\tilde{I}_{m,k+1}^{\pm \pm_{m}} 
:= \lim_{\varepsilon \to +0} \int_{0}^{\infty} e^{\pm (-1)^{m} ix^{m}} x^{k} \chi (\varepsilon x) dx 
= m^{-1} e^{\pm (-1)^{m} i\frac{\pi}{2} \frac{k+1}{m}} \varGamma \left( \frac{k+1}{m} \right). 
\notag 
\end{align} 
(Theorem \ref{th02} (ii) and Definition \ref{generalized_Fresnel_pm_pm_m}). 
This result gives an extension of the stationary phase method in one variable (Example \ref{stationary example} and Corollary \ref{cor01} (iv)). 
We note that it is known as a method for an asymptotic expansion of an oscillatory integral 
of a phase function with a {\bf non-degenerate critical point}. 

The fact above implies that we can obtain an asymptotic expansion of an oscillatory integral for a phase function with a {\bf degenerate critical point} in several variables as an extension of the stationary phase method. 

For the purpose above, we first give a summary of oscillatory integrals and the original stationary phase method 
relating to theory of asymptotic expansion in \S 2. 

In \S 3, we show existence of oscillatory integrals with positive real power phase functions used in later sections. 

In \S 4, we extend the Fresnel integrals 
by changing of a path 
for integration 
in the well-known proof using Cauchy's integral theorem. 
And then, according to oscillatory integral, 
we also obtain further generalization of the Fresnel integrals. 

Furthermore, in \S 5, 
according to generalized Fresnel integrals, we establish an asymptotic expansion of oscillatory integrals with positive real power phase functions. 

To the end of \S1, we remark 
notation which will be used in this paper: 

$\alpha = (\alpha_{1},\dots,\alpha_{n}) \in \mathbb{Z}_{\geq 0}^{n}$ 
is a multi-index with a length 
$| \alpha | = \alpha_{1} + \cdots + \alpha_{n}$, 
and then, we use 
$x^{\alpha} = x_{1}^{\alpha_{1}} \cdots x_{n}^{\alpha_{n}}$, 
$\alpha! = \alpha_{1}! \cdots \alpha_{n}!$, 
$\partial_{x}^{\alpha} 
= \partial_{x_{1}}^{\alpha_{1}} \cdots \partial_{x_{n}}^{\alpha_{n}}$ 
and 
$D_{x}^{\alpha} = D_{x_{1}}^{\alpha_{1}} \cdots D_{x_{n}}^{\alpha_{n}}$, 
where 
$\partial_{x_{j}} = \frac{\partial}{\partial x_{j}}$ 
and $D_{x_{j}} = i^{-1} \partial_{x_{j}}$ 
for $x = (x_{1}, \dots, x_{n})$. 

$C^{\infty}(\mathbb{R}^{n})$ is 
the set of complex-valued functions of class $C^{\infty}$ on $\mathbb{R}^{n}$. 
$C^{\infty}_{0}(\mathbb{R}^{n})$ is 
the set of all $f \in C^{\infty}(\mathbb{R}^{n})$ with compact support. 
$\mathcal{S}(\mathbb{R}^{n})$ is the Schwartz space of rapidly decreasing functions of class $C^{\infty}$ on $\mathbb{R}^{n}$, 
that is, the Fr\'{e}chet space of all $f \in C^{\infty}(\mathbb{R}^{n})$ 
such that $\max_{k+|\alpha| \leq m} \sup_{x \in \mathbb{R}^{n}} \langle x \rangle^{k} | \partial_{x}^{\alpha} f (x) | < \infty$ for $m \in \mathbb{Z}_{\geq 0}$, where $\langle x \rangle := (1+|x|^{2})^{1/2}$. 

$[x]$ is the Gauss' symbol for $x \in \mathbb{R}$, that is, $[x] \in \mathbb{Z}$ such that $x-1 < [x] \leq x$. 

$[x)$ is the greatest integer smaller than real number $x$, that is, $x-1 \leq [x) < x$. 

$O$ means the Landau's symbol, that is, 
$f(x) = O(g(x))~(x \to a)$ if $|f(x)/g(x)|$ 
is bounded as $x \to a$ for functions $f$ and $g$, where $a \in \mathbb{R} \cup \{ \pm \infty \}$. 

$\delta_{ij}$ is the Kronecker's delta, that is, $\delta_{ii} = 1$, and $\delta_{ij} = 0$ if $i \ne j$. 

$\tau^{+} := \max \{ \tau,0 \}$ for $\tau \in \mathbb{R}$. 

$\pm_{m} = +$ if $m$ is even, and $\pm_{m} = -$ if $m$ is odd for $m \in \mathbb{N}$, that is, $\pm_{m} 1 = (-1)^{m}$. 

\section{Preliminary}

In this section, 
we recall the oscillatory integrals and 
the original stationary phase method. 
\begin{defn} 
Let $\lambda > 0$ 
and let $\phi$ be a real-valued function of class $C^{\infty}$ on $\mathbb{R}^{n}$ 
and $a \in C^{\infty}(\mathbb{R}^{n})$. 
If there exists the following limit of improper integral: 
\begin{align} 
\tilde{I}_{\phi}[a](\lambda) 
:= Os\text{-}\int_{\mathbb{R}^{n}} e^{i \lambda \phi(x)} a (x) dx 
:= \lim_{\varepsilon \to +0} \int_{\mathbb{R}^{n}} 
e^{i \lambda \phi(x)} a (x) \chi (\varepsilon x) dx \notag 
\end{align} 
independent of $\chi \in \mathcal{S}(\mathbb{R}^{n})$ with $\chi(0) = 1$ and 
$0 < \varepsilon < 1$, 
then we call $\tilde{I}_{\phi}[a](\lambda)$ an oscillatory integral 
where we call $\phi$ 
$($resp. $a$$)$ a phase function $($resp. an amplitude function$)$. 
\end{defn} 

If we suppose a certain suitable conditions for $\phi$ and $a$, 
then we can show $\tilde{I}_{\phi}[a](\lambda)$ exists independent of 
$\chi$ and $\varepsilon$ (Theorem \ref{Lax02} (iv)). 
The fundamental properties are the following 
(cf. \cite{Kumano-go} p.47.): 
\begin{prop} 
\label{chi epsilon} 
Let $\chi \in \mathcal{S}(\mathbb{R}^{n})$ with $\chi(0) = 1$. 
Then 
\begin{enumerate} 
\item[(i)] 
$\chi(\varepsilon x) \to 1$ uniformly on any compact set in $\mathbb{R}^{n}$ as $\varepsilon \to +0$. 
\item[(ii)] 
For each multi-index $\alpha \in \mathbb{Z}_{\geq 0}^{n}$, 
there exists a positive constant $C_{\alpha}$ independent of $0 < \varepsilon < 1$ 
such that 
for any $x \in \mathbb{R}^{n}$ 
\begin{align} 
| \partial_{x}^{\alpha} (\chi(\varepsilon x)) | \leq C_{\alpha} \langle x \rangle^{-|\alpha|}. \notag 
\end{align} 
\item[(iii)] 
For any multi-index $\alpha \in \mathbb{Z}_{\geq 0}^{n}$ with $\alpha \ne 0$, 
$\partial_{x}^{\alpha} \chi(\varepsilon x) \to 0$ uniformly in $\mathbb{R}^{n}$ as $\varepsilon \to +0$. 
\end{enumerate} 
\end{prop} 

Next 
we summarize the Fourier transforms of rapidly decreasing function of class $C^{\infty}$. 
\begin{defn} 
\label{the Fourier transform} 
Let $f \in \mathcal{S}(\mathbb{R}^{n})$. 
Then 
we define by $\widehat{f}=\mathcal{F}[f]$ the Fourier transform of $f$ as 
\begin{align} 
\mathcal{F}[f](\xi) 
:= \frac{1}{(2\pi)^{\frac{n}{2}}} \int_{\mathbb{R}^{n}} 
e^{-i\langle x,\xi \rangle} f(x) dx, \notag 
\end{align} 
where $\langle x,\xi \rangle :=\sum_{k=1}^{n}x_{k} \xi _{k}$ 
for $x=(x_{1},\dots,x_{n})$ and $\xi =(\xi_{1},\dots,\xi _{n}) 
\in \mathbb{R}^{n}$. 
\end{defn} 

If $A$ is a real symmetric non-singular $n \times n$ matrix, 
then the Fourier transform of $e^{i (1/2) \langle Ax,x \rangle}$ is given 
in the following way (\cite{Hormander01}, \cite{Hormander02}, \cite{Duistermaat01}, \cite{Grigis-Sjostrand}, \cite{Fujiwara1}): 
\begin{prop} 
\label{the Fourier transform of e_ix2} 
Let $A$ be a real symmetric non-singular $n \times n$ matrix with ``$p$" 
positive and ``$n-p$" negative eigenvalues. 
\begin{enumerate} 
\item[(i)] 
If $A = \pm 1$ for $n=1$, 
then 
\begin{align} 
\mathcal{F}[e^{\pm i\frac{1}{2}x^{2}}](\xi) 
= e^{\pm i\frac{\pi}{4}} e^{\mp i\frac{1}{2} \xi^{2}}, \notag 
\end{align} 
where double signs $\pm, \mp$ are in same order. 
\item[(ii)] 
If $n \geq 1$, 
then 
\begin{align} 
\mathcal{F}[e^{i \frac{1}{2} \langle Ax,x \rangle}](\xi) 
= \frac{e^{i \frac{\pi}{4} \mathrm{sgn}A}}{|\det A|^{\frac{1}{2}}} 
e^{-i \frac{1}{2} \langle A^{-1}\xi,\xi \rangle}, \notag 
\end{align} 
where $\mathrm{sgn}A := p-(n-p)$. 
\end{enumerate}
\end{prop}

By Proposition \ref{the Fourier transform of e_ix2}, 
we can obtain an asymptotic expansion of the oscillatory integral 
with a non-degenerate quadratic phase $\phi(x) = (1/2) \langle Ax,x \rangle$ 
in the following way: 
\begin{prop} 
\label{quadratic phase} 
Suppose that $\lambda > 0$, 
$a \in \mathcal{S}(\mathbb{R}^{n})$ 
and $A$ is a real symmetric non-singular $n \times n$ matrix. 
Then there exists a positive constant $C$ such that 
for any $N \in \mathbb{N}$ and $\lambda \geq 1$, 
\begin{align} 
&\int_{\mathbb{R}^{n}} e^{i\frac{1}{2} \lambda \langle Ax,x \rangle} a(x) dx \notag \\ 
&= (2\pi)^{\frac{n}{2}} \frac{e^{i\frac{\pi}{4} 
\mathrm{sgn}A}}{|\det A|^{\frac{1}{2}}} 
\sum_{k=0}^{N-1} \frac{1}{k!} 
\Big( -i\frac{1}{2} \langle A^{-1} D_{x},D_{x} \rangle \Big)^{k} \Big|_{x=0} 
a(x) \lambda^{-k-\frac{n}{2}} + R_{N}(\lambda) \notag 
\end{align} 
and 
\begin{align} 
|R_{N}(\lambda)|
\leq (2\pi)^{\frac{n}{2}} \frac{C}{|\det A|^{N+\frac{1}{2}}} 
\frac{1}{N!} \Big( \sum_{|\alpha| \leq 2(N+n)} 
\int_{\mathbb{R}^{n}} |\partial_{x}^{\alpha} a(x)| dx \Big) 
\lambda^{-N-\frac{n}{2}}. \notag 
\end{align} 
\end{prop}

\begin{exmp} 
\label{stationary example} 
If $n=1$ and $A=\pm 2$, 
since 
$\mathrm{sgn}A = \pm 1$, 
then for any $N \in \mathbb{N}$, 
\begin{align} 
&\int_{-\infty}^{\infty} e^{\pm i\lambda x^{2}} a(x) dx \notag \\ 
&= \sqrt{\pi} e^{\pm i\frac{\pi}{4}} \sum_{k=0}^{N-1} \frac{1}{k!} \left( \frac{e^{\pm i\frac{\pi}{2}}}{4} \frac{d^{2}}{dx^{2}} \right)^{k} \bigg|_{x=0} a(x) \lambda^{-k-\frac{1}{2}} + O\left( \lambda ^{-N-\frac{1}{2}} \right), 
\notag 
\end{align} 
as $\lambda \to \infty$, where double signs $\pm$ are in same order. 
\end{exmp} 

In order to treat more general cases of the phase function, 
we prepare the following two lemmas. 
The first one is the Morse lemma 
(\cite{Milnor}, \cite{Duistermaat01}, \cite{Grigis-Sjostrand}, \cite{Fujiwara1}). 

\begin{lem} 
Let $\phi$ be a real-valued function of class $C^{\infty}$ on 
a neighborhood of $\bar{x}$ in $\mathbb{R}^{n}$  
such that $\bar{x}$ is an only non-degenerate critical point of $\phi$, 
that is, 
if and only if $\nabla \phi (\bar{x})=0$, 
and $\det \mathrm{Hess} \phi (\bar{x}) \ne 0$. 
Then 
there exist neighborhoods $U$ of $\bar{x}$ and $V$ of $0$ in $\mathbb{R}^{n}$, 
and $C^{\infty}$ diffeomorphism $\varPhi : V \longrightarrow U$ 
such that $x=\varPhi (y)$ for $x=(x_{1},\dots,x_{n}) \in U$ and $y=(y_{1},\dots,y_{n}) \in V$ 
and 
\begin{align} 
\phi (x) - \phi (\bar{x}) 
= \frac{1}{2} (y_{1}^{2} + \cdots + 
y_{p}^{2} - y_{p+1}^{2} - \cdots - y_{n}^{2}), \notag 
\end{align} 
where $\mathrm{Hess} \phi (\bar{x}) 
:= (\partial ^{2} 
\phi (\bar{x})/\partial x_{i} \partial x_{j})_{i,j=1,\dots,n}$ 
is Hessian matrix of $\phi$ 
at $\bar{x}$ with ``$p$" positive and ``$n-p$" negative eigenvalues. 
\end{lem}

The second one is for an estimation of the remainder of an asymptotic expansion for an oscillatory integral (\cite{Hormander02}, \cite{Fujiwara1}). 
\begin{lem} 
\label{Lax's technique} 
Let $\lambda > 0$, 
$a \in \mathcal{S}(\mathbb{R}^{n})$ 
and $\phi$ a real-valued function of class $C^{\infty}$ on $\mathbb{R}^{n}$ 
with $| \nabla \phi (x) | \geq d > 0$ for $x \in \mathrm{supp} a$. 
Then 
for each $N \in \mathbb{N}$, 
there exists a positive constant $C_{N}$ such that for any $\lambda \geq 1$, 
\begin{align} 
\left| \int_{\mathbb{R}^{n}} e^{i \lambda \phi(x)} a (x) dx \right| 
\leq C_{N}(\lambda d^{2})^{-N}. \notag 
\end{align} 
\end{lem} 

We are now in a position to state the stationary phase method 
(\cite{Hormander01}, \cite{Hormander02}, 
\cite{Duistermaat01}, \cite{Fujiwara1}). 
\begin{thm} 
\label{th00} 
Suppose that $\lambda > 0$, $a \in \mathcal{S}(\mathbb{R}^{n})$ 
and $\phi$ is a real-valued function of class $C^{\infty}$ on a neighborhood of 
$\bar{x}$ in $\mathbb{R}^{n}$  
such that $\bar{x}$ is an only non-degenerate critical point of $\phi$. 
Then 
there exist neighborhoods $U$ of $\bar{x}$ 
and $V$ of $0$ in $\mathbb{R}^{n}$, 
and $C^{\infty}$ diffeomorphism 
$\varPhi : V \longrightarrow U$ such that $x=\varPhi (y)$ 
for $x =(x_{1},\dots,x_{n}) \in U$ and $y =(y_{1},\dots,y_{n}) \in V$, 
and for each $N \in \mathbb{N}$, 
there exists a positive constant $C_{N}$ such that for any $\lambda \geq 1$, 
\begin{align} 
&\int_{\mathbb{R}^{n}} e^{i\lambda \phi (x)} a(x) dx 
= (2\pi)^{\frac{n}{2}} \frac{e^{i \frac{\pi}{4} \mathrm{sgn} \mathrm{Hess} \phi (\bar{x})}}{|\det \mathrm{Hess} \phi (\bar{x})|^{\frac{1}{2}}} e^{i\lambda \phi(\bar{x})} \notag \\ 
&\times \sum_{k=0}^{N-1} \frac{1}{k!} \Big( -i\frac{1}{2} \langle \mathrm{Hess} \phi (\bar{x})^{-1} D_{y},D_{y} \rangle \Big)^{k} 
\Big|_{y=0} \{ (a \circ \varPhi) J_{\varPhi} \} (y) \lambda^{-k-\frac{n}{2}} + R_{N}(\lambda) \notag 
\end{align} 
and 
\begin{align} 
|R_{N}(\lambda)| \leq C_{N} \lambda^{-N-\frac{n}{2}}, \notag 
\end{align} 
where 
$J_{\varPhi}(y) := 
\det (\partial x_{j}/\partial y_{k})_{j,k=1,\dots,n}$ is a Jacobian of $\varPhi$. 
\end{thm} 

\section{Existence of Oscillatory Integrals} 

In this section, 
we shall show existence of oscillatory integrals used in later sections. 
First we define the class of amplitude functions as follows 
(cf. \cite{Kumano-go} p.46.): 
\begin{defn} 
\label{A_tau_delta} 
Assume that $p>0$. 
Let $\tau \in \mathbb{R}$ and $-1 \leq \delta < p-1$. 
We say that $a \in C^{\infty}(\mathbb{R})$ belongs to  the class $\mathcal{A}^{\tau}_{\delta}(\mathbb{R})$ if 
for each $k \in \mathbb{Z}_{\geq 0}$, there exists a positive constant $C_{k}$ such that for any $x \in \mathbb{R}$ 
\begin{align} 
|a^{(k)}(x)| \leq C_{k} \langle x \rangle^{\tau + \delta k}. 
\notag  
\end{align} 
Then for any $l \in \mathbb{Z}_{\geq 0}$, we set 
\begin{align} 
|a|^{(\tau)}_{l} 
:= \max_{k=0,\dots,l} \sup_{x \in \mathbb{R}} \langle x \rangle^{-\tau - \delta k} |a^{(k)}(x)|. 
\notag 
\end{align} 
And then 
\begin{align} 
|a^{(k)}(x)| \leq |a|^{(\tau)}_{l} \langle x \rangle^{\tau + \delta k}. 
\label{A_tau_delta_def02} 
\end{align} 
\end{defn} 

\begin{rem} \label{rem01} 
We see the following immediately: 
\begin{align} 
\mathcal{S}(\mathbb{R}) 
= \bigcap_{\tau \leq 0} \mathcal{A}^{\tau}_{0}(\mathbb{R}), 
\notag 
\end{align} 
and if $a \in \mathcal{A}^{\tau}_{\delta}(\mathbb{R})$, then for any $j \in \mathbb{Z}_{\geq 0}$, $a^{(j)} \in \mathcal{A}^{\tau + \delta j}_{\delta}(\mathbb{R})$ and for any $l \in \mathbb{Z}_{\geq 0}$, 
\begin{align} 
|a^{(j)}|^{(\tau + \delta j)}_{l} 
&= \max_{k=0,\dots,l} \sup_{x \in \mathbb{R}} \langle x \rangle^{-(\tau + \delta j) - \delta k} |a^{(j+k)}(x)| \notag \\ 
&= \max_{k'=j,\dots,j+l} \sup_{x \in \mathbb{R}} \langle x \rangle^{-\tau - \delta k'} |a^{(k')}(x)| 
\leq |a|^{(\tau)}_{j+l}. 
\label{A_tau_delta_def03} 
\end{align} 
\end{rem} 

Then we have following lemma: 
\begin{lem} 
\label{L_star_l} 
Assume that $\lambda > 0$, $p > 0$ and $q > 0$. 
Let $a \in \mathcal{A}^{\tau}_{\delta}(\mathbb{R})$, 
$\varphi \in C^{\infty}_{0}(\mathbb{R})$ a cutoff function such that $\varphi \equiv 1$ on $|x| \leq 1$ and $\varphi \equiv 0$ on $|x| \geq r > 1$, 
$\psi_{h} := 1 - \delta_{h1} \varphi$ for $h=0,1$, 
and $a_{h} := a \psi_{h}$, 
and let $\chi \in \mathcal{S}(\mathbb{R})$ with $\chi(0) = 1$, $0 < \varepsilon < 1$ and $\chi_{\varepsilon}(x) := \chi(\varepsilon x)$ for $x \in \mathbb{R}$, 
and let $L^{*} := - \frac{1}{i\lambda} \frac{d}{dx}  \frac{1}{px^{p-1}}$ be the formal adjoint operator of $L := \frac{1}{px^{p-1}} \frac{1}{i\lambda} \frac{d}{dx}$, and $l_{0}:=[q/p)$.
Then for any $k \in \mathbb{Z}_{\geq 0}$, the following hold: 

\begin{enumerate} 
\item[(i)] 
For each $l \in \mathbb{Z}_{\geq 0}$, there exist real constants $C_{l,j}$ for $j=0,\dots,l$ 
such that for any $x \in (0,\infty)$ and $h=0,1$, 
\begin{align} 
L^{*l} (x^{q-1} a_{h}(x) \chi_{\varepsilon}^{(k)}(x)) 
= \left( \frac{i}{\lambda p} \right)^{l} \sum_{j=0}^{l} C_{l,j} x^{q-1-pl+j} (a_{h}(x) \chi_{\varepsilon}^{(k)}(x))^{(j)}, 
\label{L star_l_01} 
\end{align} 
where $L^{*0}$ is an identity operator, and $C_{l,j} = (q-pl+j) C_{l-1,j} + C_{l-1,j-1}$ for $1 \leq j \leq l-1$, $C_{l,0} = (q-pl) C_{l-1,0}$ and $C_{l,l} = C_{l-1,l-1}$ for $l \in \mathbb{N}$, and $C_{0,0} = 1$. 
Then $C_{l,0} = \prod_{s=1}^{l} (q-ps)$ for $l \in \mathbb{N}$ and $C_{l,l} = 1$ for $l \in \mathbb{Z}_{\geq 0}$. 
Note $q-1-pl+j = (q-1)-(p-1)l-(l-j)$. 
\item[(ii)] 
If $q>p$ and $h=0$, then for any $l=0,\dots,l_{0}$, or if $p > 0$, $q > 0$ and $h = 1$, then for any $l \in \mathbb{Z}_{\geq 0}$, 
the following improper integrals are absolutely convergent:
\begin{align} 
\int_{0}^{\infty} e^{\pm i\lambda x^{p}} (\pm L^{*})^{l} (x^{q-1} a_{h}(x) \chi_{\varepsilon}^{(k)}(x)) dx. 
\notag 
\end{align} 
\item[(iii)] 
If $q>p$ and $h=0$, then for any $l=1,\dots,l_{0}$, or if $p > 0$, $q > 0$ and $h = 1$, then for any $l \in \mathbb{N}$, as $x \to +0$ or $x \to \infty$, 
\begin{align} 
\big| e^{\pm i\lambda x^{p}} (\pm i\lambda px^{p-1})^{-1} (\pm L^{*})^{l-1} (x^{q-1} a_{h}(x) \chi_{\varepsilon}^{(k)}(x)) \big| \to 0. 
\notag 
\end{align} 
\item[(iv)] 
If $q>p$ and $h=0$, then for any $l=1,\dots,l_{0}$, or if $p > 0$, $q > 0$ and $h = 1$, then for any $l \in \mathbb{N}$, 
\begin{align} 
\int_{0}^{\infty} e^{\pm i\lambda x^{p}} x^{q-1} a_{h}(x) \chi_{\varepsilon}^{(k)}(x) dx 
= \int_{0}^{\infty} e^{\pm i\lambda x^{p}} (\pm L^{*})^{l} (x^{q-1} a_{h}(x) \chi_{\varepsilon}^{(k)}(x)) dx, 
\label{I_pq_pm_a_lambda} 
\end{align} 
\end{enumerate} 
where double signs $\pm$ are in same order. 
\end{lem} 

\begin{proof} 
Since the lower side of double signs $\pm$ can be obtained as the conjugate of the upper one, 
we shall show the upper one. 

(i) 
By induction on $l \in \mathbb{Z}_{\geq 0}$. 
If $l=0$, then we have 
\begin{align} 
L^{*0} (x^{q-1} a_{h}(x) \chi_{\varepsilon}^{(k)}(x)) 
&= x^{q-1} a_{h}(x) \chi_{\varepsilon}^{(k)}(x) \notag \\ 
&= \left( \frac{i}{\lambda p} \right)^{0} \sum_{j=0}^{0} C_{0,j} x^{q-1-p \cdot 0+j} (a_{h}(x) \chi_{\varepsilon}^{(k)}(x))^{(j)} 
\notag 
\end{align} 
for $x \in (0,\infty)$ and $h=0,1$, where $C_{0,0} = 1$. Thus \eqref{L star_l_01} holds for $l=0$.

Next if \eqref{L star_l_01} holds for $l-1$ with $l \geq 1$, 
then we have 
\begin{align} 
&L^{*l} (x^{q-1} a_{h}(x) \chi_{\varepsilon}^{(k)}(x)) 
= L^{*} L^{*l-1} (x^{q-1} a_{h}(x) \chi_{\varepsilon}^{(k)}(x)) \notag \\ 
&= - \frac{1}{i\lambda} \frac{d}{dx} \frac{1}{px^{p-1}} \left( \frac{i}{\lambda p} \right)^{l-1} \sum_{j=0}^{l-1} C_{l-1,j} x^{q-1-p(l-1)+j} (a_{h}(x) \chi_{\varepsilon}^{(k)}(x))^{(j)} \notag \\ 
&= \left( \frac{i}{\lambda p} \right)^{l} \sum_{j=0}^{l-1} C_{l-1,j} \frac{d}{dx} \left\{ x^{q-pl+j} (a_{h}(x) \chi_{\varepsilon}^{(k)}(x))^{(j)} \right\} 
= \left( \frac{i}{\lambda p} \right)^{l} \sum_{j=0}^{l-1} C_{l-1,j} \notag \\ 
&\hspace{0.75cm}\times \left\{ (q-pl+j) x^{q-1-pl+j} (a_{h}(x) \chi_{\varepsilon}^{(k)}(x))^{(j)} + x^{q-pl+j} (a_{h}(x) \chi_{\varepsilon}^{(k)}(x))^{(j+1)} \right\} \notag \\ 
&= \left( \frac{i}{\lambda p} \right)^{l} \sum_{j=0}^{l} C_{l,j} x^{q-1-pl+j} (a_{h}(x) \chi_{\varepsilon}^{(k)}(x))^{(j)} \notag 
\end{align} 
for $x \in (0,\infty)$ and $h=0,1$, 
where $C_{l,j} = (q-pl+j) C_{l-1,j} + C_{l-1,j-1}$ for $1 \leq j \leq l-1$, $C_{l,0} = (q-pl) C_{l-1,0}$ and $C_{l,l} = C_{l-1,l-1}$ for $l \in \mathbb{N}$. 
Thus \eqref{L star_l_01} holds for $l$. 
And then $C_{l,0} = \prod_{s=1}^{l} (q-ps)$ for $l \in \mathbb{N}$ and $C_{l,l} = 1$ for $l \in \mathbb{Z}_{\geq 0}$. 

(ii) 
Put $f_{h,l}(x) = e^{i\lambda x^{p}} L^{*l} (x^{q-1} a_{h}(x) \chi_{\varepsilon}^{(k)}(x))$ for $h=0,1$ and $l \in \mathbb{Z}_{\geq 0}$. 
Then since $a_{h}(x) := a(x) \psi_{h}(x)$, by \eqref{L star_l_01} and Leibniz's formula, 
\begin{align} 
f_{h,l}(x) 
= e^{i\lambda x^{p}} \left( \frac{i}{\lambda p} \right)^{l} \sum_{j=0}^{l} C_{l,j} x^{q-1-pl+j} \sum_{s+t+u=j} \frac{j!}{s!t!u!} a^{(s)}(x) \psi_{h}^{(t)}(x) \chi_{\varepsilon}^{(k+u)}(x). 
\label{f_m(x)} 
\end{align} 
Since $f_{h,l}$ is continuous on $(0,\infty)$, 
then $f_{h,l}$ is integrable on $[u,1]$ for any $u \in (0,1]$, and integrable on $[1,v]$ for any $v \in [1,\infty)$. 

By \eqref{A_tau_delta_def02}, for any $x \in (0,\infty)$, 
\begin{align} 
|f_{h,l}(x)| 
&= |L^{*l} (x^{q-1} a_{h}(x) \chi_{\varepsilon}^{(k)}(x))| \notag \\ 
&\leq (\lambda p)^{-l} \sum_{j=0}^{l} |C_{l,j}| |x|^{q-1-pl+j} \sum_{s+t+u=j} \frac{j!}{s!t!u!} |a^{(s)}(x)| |\psi_{h}^{(t)}(x)| |\chi_{\varepsilon}^{(k+u)}(x)| \notag \\  
&\leq (\lambda p)^{-l} \sum_{j=0}^{l} |C_{l,j}| |x|^{q-1-pl+j} \notag \\ 
&\hspace{0.5cm}\times \sum_{s+t+u=j} \frac{j!}{s!t!u!} |a|^{(\tau)}_{l} \langle x \rangle^{\tau + \delta s} \langle x \rangle^{-t} \langle x \rangle^{t} |\psi_{h}^{(t)}(x)| |\chi_{\varepsilon}^{(k+u)}(x)|. 
\label{|f_{h,l}(x)|02} 
\end{align} 
Here since $\psi_{h} \in C^{\infty}(\mathbb{R})$ such that 
$\psi_{h}^{(t)} \equiv \delta_{0t}$ on $|x| \geq r > 1$, 
\begin{align} 
\langle x \rangle^{t} |\psi_{h}^{(t)}(x)| 
&\leq 
\begin{cases} 
\displaystyle 
\max_{t=0,\dots,l} \sup_{|x|< r} \langle x \rangle^{t} |\psi_{h}^{(t)}(x)|^{\delta_{1h}} & \text{for $|x| < r$}, \\ 
1 & \text{for $|x|\geq r$} 
\end{cases} 
\notag \\ 
&\leq \max_{t=0,\dots,l} \sup_{|x|< r} \langle x \rangle^{t} |\psi_{h}^{(t)}(x)|^{\delta_{1h}}+1 
=: |\psi_{h}|^{(l)}_{r}. 
\label{|psi_{h}|_{l,2}} 
\end{align} 
Hence by \eqref{|f_{h,l}(x)|02} and \eqref{|psi_{h}|_{l,2}}, since $-1 \leq \delta$, 
\begin{align} 
|f_{h,l}(x)| 
&= |L^{*l} (x^{q-1} a_{h}(x) \chi_{\varepsilon}^{(k)}(x))| 
\leq (\lambda p)^{-l} \sum_{j=0}^{l} |C_{l,j}| |x|^{q-1-pl+j} \notag \\ 
&\hspace{0.5cm}\times \sum_{s+t+u=j} \frac{j!}{s!t!u!} |a|^{(\tau)}_{l} \langle x \rangle^{\tau + \delta (s+t)} |\psi_{h}|^{(l)}_{r} |\chi_{\varepsilon}^{(k+u)}(x)|. 
\label{|f_{h,l}(x)|} 
\end{align} 
Here for any $|x| \geq 1$, 
\begin{align} 
&|x| = (|x|^{2})^{1/2} \leq (1+|x|^{2})^{1/2} 
\leq (|x|^{2}+|x|^{2})^{1/2} = 2^{1/2} |x|. 
\notag 
\end{align} 
Then since $|x| \leq \langle x \rangle \leq 2^{1/2} \langle x \rangle$ and $2^{-1/2} \langle x \rangle \leq |x|$, 
for any $\tau \in \mathbb{R}$, 
\begin{align} 
|x|^{\tau} \leq 2^{|\tau|/2} \langle x \rangle^{\tau}. 
\label{|x|_angle_x} 
\end{align} 
And since $\chi \in \mathcal{S}(\mathbb{R})$, for each $k \in \mathbb{Z}_{\geq 0}$ and for each $u \in \mathbb{Z}_{\geq 0}$, there exists a positive constant $\tilde{C}_{k+u}$ such that for any $0 < \varepsilon < 1$, for any $m \in \mathbb{N}$ and for any $x \in \mathbb{R}$, 
\begin{align} 
|\chi_{\varepsilon}^{(k+u)}(x)| 
&= |\partial_{x}^{k+u} (\chi (\varepsilon x))| 
= |\partial_{y}^{k+u} \chi (\varepsilon x) \varepsilon^{k+u}| 
\leq |\partial_{y}^{k+u} \chi (\varepsilon x)| \notag \\ 
&\leq \tilde{C}_{k+u} \langle \varepsilon x \rangle^{-m} 
= \tilde{C}_{k+u} (1+|\varepsilon x|^{2})^{-m/2} 
= \tilde{C}_{k+u} \varepsilon^{-m} (\varepsilon^{-2}+|x|^{2})^{-m/2} \notag \\ 
&\leq \tilde{C}_{k+u} \varepsilon^{-m} (1+|x|^{2})^{-m/2} 
= \tilde{C}_{k+u} \varepsilon^{-m} \langle x \rangle^{-m}. 
\label{chi_est00} 
\end{align} 
Hence by \eqref{|f_{h,l}(x)|}, \eqref{|x|_angle_x} and \eqref{chi_est00}, since $-1 \leq \delta$, for any $|x| \geq 1$, 
\begin{align} 
|f_{h,l}(x)| 
&\leq (\lambda p)^{-l} \sum_{j=0}^{l} |C_{l,j}| 2^{|q-1-pl+j|/2} \langle x \rangle^{q-1-pl+j} \notag \\ 
&\hspace{0.5cm}\times \sum_{s+t+u=j} \frac{j!}{s!t!u!} |a|^{(\tau)}_{l} \langle x \rangle^{\tau + \delta (s+t)} \langle x \rangle^{-u} \langle x \rangle^{u} |\psi_{h}|^{(l)}_{r} \tilde{C}_{k+u} \varepsilon^{-m} \langle x \rangle^{-m} \notag \\ 
&\leq (\lambda p)^{-l} \sum_{j=0}^{l} \max_{j=0,\dots,l} |C_{l,j}| 2^{|q-1-pl+j|/2} \langle x \rangle^{q-1-pl+l} \notag \\ 
&\hspace{0.5cm}\times \sum_{s+t+u=j} \frac{j!}{s!t!u!} |a|^{(\tau)}_{l} \langle x \rangle^{\tau + \delta l} \langle x \rangle^{u} |\psi_{h}|^{(l)}_{r} \max_{u=0,\dots,l} \tilde{C}_{k+u} \varepsilon^{-m} \langle x \rangle^{-m} \notag \\ 
&\leq (\lambda p)^{-l} \max_{j=0,\dots,l} |C_{l,j}| 2^{|q-1-pl+j|/2} \sum_{j=0}^{l} 3^{j} |a|^{(\tau)}_{l} |\psi_{h}|^{(l)}_{r} \max_{u=0,\dots,l} \tilde{C}_{k+u} \varepsilon^{-m} \notag \\ 
&\hspace{0.5cm}\times \langle x \rangle^{q-1+\tau-(p-1-\delta)l + l -m} \notag \\ 
&= C^{(k)}_{l,\varepsilon} \lambda^{-l} |a|^{(\tau)}_{l} \langle x \rangle^{q-1+\tau-(p-1 - \delta)l +l-m} \notag \\ 
&\leq C^{(k)}_{l,\varepsilon} \lambda^{-l} |a|^{(\tau)}_{l} \langle x \rangle^{(q+\tau)^{+}-m}, 
\notag 
\end{align} 
where 
\begin{align} 
C^{(k)}_{l,\varepsilon} 
:&= p^{-l} \max_{j=0,\dots,l} |C_{l,j}| 2^{|q-1-pl+j|/2} \cdot \frac{3^{l+1}-1}{2} |\psi_{h}|^{(l)}_{r} \max_{u=0,\dots,l} \tilde{C}_{k+u} \varepsilon^{-m}. 
\notag 
\end{align} 
Hence 
\begin{align} 
f_{h,l}(x) 
= O(x^{\beta})~(x \to \infty) 
\label{f_l(x)_order} 
\end{align} 
with $\beta = (q+ \tau)^{+} -m$ for any $m \in \mathbb{Z}_{\geq 0}$. 
Here let $m = [(q+ \tau)^{+}]+2$. 
Since $x-[x]-1<0$ for $x \in \mathbb{R}$, then $\beta <-1$. 
Thus $\int_{1}^{\infty} f_{h,l}(x) dx$ is absolutely convergent. 

When $q>p$ and $h=0$, then $a_{h} = a$, and let $l_{0}:= [q/p)$. 
Since $(q/p)-1 \leq l_{0} < q/p$, then $0 < q-pl_{0} \leq p$. 
Hence by \eqref{f_m(x)}, for any $l=0,\dots,l_{0}$, 
\begin{align} 
f_{h,l}(x) 
= \sum_{j=0}^{l} O(x^{q-1-pl+j}) 
= O(x^{q-1-pl}) 
= O(x^{\alpha}) ~(x \to +0) 
\label{f_m_orger_alpha} 
\end{align} 
with $\alpha = q-1-pl \geq q-1-pl_{0} > -1$. 

When $p > 0$, $q > 0$ and $h = 1$, since $a_{1} = a(1 - \varphi)$, then $f_{1,l} \equiv 0$ on $(0,1]$ for $l \in \mathbb{Z}_{\geq 0}$. 
Hence $\int_{0}^{1} f_{h,l}(x) dx$ is absolutely convergent in either case. 

Therefore the following improper integral is absolutely convergent: 
\begin{align} 
\int_{0}^{\infty} f_{h,l}(x) dx 
= \int_{0}^{1} f_{h,l}(x) dx + \int_{1}^{\infty} f_{h,l}(x) dx. 
\notag 
\end{align} 

(iii) 
Put $g_{h,l-1}(x) = (i\lambda px^{p-1})^{-1} f_{h,l-1}(x)$ for for $h=0,1$ and $l \in \mathbb{N}$. 
Then by \eqref{f_l(x)_order}, 
\begin{align} 
g_{h,l-1}(x) 
&= O(x^{-(p-1)}) O(x^{(q+ \tau)^{+} -m}) 
= O(x^{\beta'})~(x \to \infty) 
\notag 
\end{align} 
with $\beta' = (q+\tau)^{+} -m$ for any $m \in \mathbb{Z}_{\geq 0}$. 
Here let $m = [(q+\tau)^{+}]+1$. Since $x-[x]-1<0$ for $x \in \mathbb{R}$, then $\beta' <0$. 

When $q>p$ and $h=0$, by \eqref{f_m_orger_alpha}, for any $l=1,\dots,l_{0}$, 
\begin{align} 
g_{h,l-1}(x) 
= O(x^{-(p-1)}) O(x^{q-1-p(l-1)}) 
= O(x^{\alpha'}) ~(x \to +0) 
\notag 
\end{align} 
with $\alpha' = q-pl \geq q-pl_{0} > 0$. 

When $p > 0$, $q > 0$ and $h = 1$, by (ii), since $f_{l-1} \equiv 0$, then $g_{l-1} \equiv 0$ on $(0,1]$ for $l \in \mathbb{N}$. 
Therefore, as $x \to \infty$ or $x \to +0$, 
\begin{align} 
|g_{h,l-1}(x)| \to 0. 
\notag 
\end{align} 

(iv) 
By induction on $l \in \mathbb{N}$. 
If $l=1$, then since $L(e^{i\lambda x^{p}}) = e^{i\lambda x^{p}}$ when $x \ne 0$, 
by integration by parts, (ii) and (iii), 
\begin{align} 
&\int_{0}^{\infty} e^{i\lambda x^{p}} x^{q-1} a_{h}(x) \chi_{\varepsilon}^{(k)}(x) dx 
= \lim_{\substack{u \to +0\\v \to \infty}} \int_{u}^{v} \frac{1}{px^{p-1}} \frac{1}{i\lambda} \frac{d}{dx}(e^{i\lambda x^{p}}) x^{q-1} a_{h}(x) \chi_{\varepsilon}^{(k)}(x) dx \notag \\ 
&= \lim_{\substack{u \to +0\\v \to \infty}} \bigg\{ \Big[ e^{i\lambda x^{p}} (i\lambda p x^{p-1})^{-1} x^{q-1} a_{h}(x) \chi_{\varepsilon}^{(k)}(x) \Big]_{u}^{v} \bigg. \notag \\ 
&\hspace{0.65cm}\bigg. + \int_{u}^{v} e^{i\lambda x^{p}} L^{*} (x^{q-1} a_{h}(x) \chi_{\varepsilon}^{(k)}(x)) dx \bigg\} 
= \int_{0}^{\infty} e^{i\lambda x^{p}} L^{*} (x^{q-1} a_{h}(x) \chi_{\varepsilon}^{(k)}(x)) dx. 
\notag 
\end{align} 
Thus \eqref{I_pq_pm_a_lambda} holds for $l=1$. 
Next if \eqref{I_pq_pm_a_lambda} holds for $l-1$ with $l \geq 2$, 
then similarly by integration by parts, (ii) and (iii), 
\begin{align} 
&\int_{0}^{\infty} e^{i\lambda x^{p}} x^{q-1} a_{h}(x) \chi_{\varepsilon}^{(k)}(x) dx 
= \int_{0}^{\infty} e^{i\lambda x^{p}} L^{*l-1} (x^{q-1} a_{h}(x) \chi_{\varepsilon}^{(k)}(x)) dx \notag \\ 
&= \lim_{\substack{u \to +0\\v \to \infty}} \int_{u}^{v} \frac{1}{px^{p-1}} \frac{1}{i\lambda} \frac{d}{dx}(e^{i\lambda x^{p}}) L^{*l-1} (x^{q-1} a_{h}(x) \chi_{\varepsilon}^{(k)}(x)) dx \notag \\ 
&= \lim_{\substack{u \to +0\\v \to \infty}} \bigg\{ \Big[ e^{i\lambda x^{p}} (i\lambda px^{p-1})^{-1} L^{*l-1} (x^{q-1} a_{h}(x) \chi_{\varepsilon}^{(k)}(x)) \Big]_{u}^{v} \bigg. \notag \\ 
&\hspace{0.65cm}\bigg. + \int_{u}^{v} e^{i\lambda x^{p}} L^{*l} (x^{q-1} a_{h}(x) \chi_{\varepsilon}^{(k)}(x)) dx \bigg\} 
= \int_{0}^{\infty} e^{i\lambda x^{p}} L^{*l} (x^{q-1} a_{h}(x) \chi_{\varepsilon}^{(k)}(x)) dx. 
\notag 
\end{align} 
This completes the proof. 
\end{proof} 

By Lemma \ref{L_star_l}, 
we obtain the following theorem: 
\begin{thm} 
\label{Lax02} 
Assume that $\lambda > 0$, $p > 0$ and $q > 0$. 
Let $a \in \mathcal{A}^{\tau}_{\delta}(\mathbb{R})$, 
$\varphi \in C^{\infty}_{0}(\mathbb{R})$ a cutoff function such that $\varphi \equiv 1$ on $|x| \leq 1$ and $\varphi \equiv 0$ on $|x| \geq r > 1$, 
$\psi := 1 - \varphi$, 
and let $\chi \in \mathcal{S}(\mathbb{R})$ with $\chi(0) = 1$, $0 < \varepsilon < 1$ and $\chi_{\varepsilon}(x) := \chi(\varepsilon x)$ for $x \in \mathbb{R}$, 
and let $L^{*} := - \frac{1}{i\lambda} \frac{d}{dx}  \frac{1}{px^{p-1}}$ be the formal adjoint operator of $L := \frac{1}{px^{p-1}} \frac{1}{i\lambda} \frac{d}{dx}$, $l_{0}:=[q/p)$, 
and 
\begin{align} 
l_{p,q} := \left[ \frac{(q+\tau)^{+}}{p-1-\delta} \right] +1. 
\notag 
\end{align} 
Then the following hold: 
\begin{enumerate} 
\item[(i)] 
For each $k \in \mathbb{Z}_{\geq 0}$, 
there exist the following limit of improper integrals independent of $\chi_{\varepsilon}$, and the following holds: 
\begin{align} 
\lim_{\varepsilon \to +0} \int_{0}^{\infty} e^{\pm i\lambda x^{p}} x^{q-1} a(x) \varphi(x) \chi_{\varepsilon}^{(k)}(x) dx 
= \delta_{k0} \int_{0}^{\infty} e^{\pm i\lambda x^{p}} x^{q-1} a(x) \varphi(x) dx. 
\notag 
\end{align} 
\item[(ii)] 
For each $k \in \mathbb{Z}_{\geq 0}$, 
there exist the following limit of improper integrals independent of $\chi_{\varepsilon}$, 
and for any $l \in \mathbb{N}$ such that $l \geq l_{p,q}$, the following holds: 
\begin{align} 
&\lim_{\varepsilon \to +0} \int_{0}^{\infty} e^{\pm i\lambda x^{p}} x^{q-1} a(x) \psi(x) \chi_{\varepsilon}^{(k)}(x) dx 
= \delta_{k0} \int_{0}^{\infty} e^{\pm i\lambda x^{p}} (\pm L^{*})^{l} (x^{q-1} a(x) \psi(x)) dx. 
\notag 
\end{align} 
\item[(iii)] 
If $k \ne 0$, then 
\begin{align} 
\lim_{\varepsilon \to +0} \int_{0}^{\infty} e^{\pm i\lambda x^{p}} x^{q-1} a(x) \chi_{\varepsilon}^{(k)}(x) dx 
= 0. \notag 
\end{align} 
\item[(iv)] 
There exist the following oscillatory integrals, 
and for any $l \in \mathbb{N}$ such that $l \geq l_{p,q}$, the following holds: 
\begin{align} 
\tilde{I}_{p,q}^{\pm}[a](\lambda) 
:&= Os\text{-}\int_{0}^{\infty} e^{\pm i\lambda x^{p}} x^{q-1} a(x) dx 
:= \lim_{\varepsilon \to +0} \int_{0}^{\infty} e^{\pm i\lambda x^{p}} x^{q-1} a(x) \chi_{\varepsilon}(x) dx \notag \\ 
&= \int_{0}^{\infty} e^{i\lambda x^{p}} x^{q-1} a(x) \varphi(x) dx + \int_{0}^{\infty} e^{i\lambda x^{p}} L^{*l} (x^{q-1} a(x) \psi(x)) dx. 
\notag 
\end{align} 
Then for each $l \in \mathbb{N}$ such that $l \geq l_{p,q}$, there exists a positive constant $C_{l}$ such that for any $\lambda > 0$, 
\begin{align} 
|\tilde{I}_{p,q}^{\pm}[a](\lambda)| \leq C_{l} |a|^{(\tau)}_{l}, 
\notag 
\end{align} 
where $|a|^{(\tau)}_{l} := \max_{k=0,\dots, l} \sup_{x \in \mathbb{R}} \langle x \rangle^{-\tau - \delta k} |a^{(k)}(x)|$. 
\item[(v)] 
If $q > p$, 
then there exists a positive constant $C_{p,q}$ such that for any $\lambda \geq 1$, 
\begin{align} 
| \tilde{I}_{p,q}^{\pm}[a](\lambda) | 
\leq C_{p,q} |a|^{(\tau)}_{l_{0}+l_{p,q}} \lambda^{-\frac{q-p}{p}}, 
\notag 
\end{align} 
\end{enumerate} 
where double signs $\pm$ are in same order. 
\end{thm} 

\begin{proof} 
Since the lower side of double signs $\pm$ can be obtained as the conjugate of the upper one, 
we shall show the upper one. 

(i) Put $f_{k}(x) = e^{i\lambda x^{p}} x^{q-1} a(x) \varphi(x) \chi_{\varepsilon}^{(k)}(x)$ for $k \in \mathbb{Z}_{\geq 0}$. 
Since $f_{k}$ is continuous on $(0,\infty)$ and $f_{k} \equiv 0$ on $[r,\infty)$, 
then $f_{k}$ is integrable on $[u,\infty)$ for any $u \in (0,\infty)$ and $f_{k}(x) = O(x^{\alpha})~(x \to +0)$ with $\alpha =q-1>-1$. 
Thus for any $k \in \mathbb{Z}_{\geq 0}$, the following improper integral is absolutely convergent: 
\begin{align} 
\int_{0}^{\infty} e^{i\lambda x^{p}} x^{q-1} a(x) \varphi(x) \chi_{\varepsilon}^{(k)}(x) dx. 
\label{tilde_I_phi_a_varphi02} 
\end{align} 
In order to apply Lebesgue's convergence theorem, 
we shall show \eqref{tilde_I_phi_a_varphi02} is bounded independent of $\chi_{\varepsilon}$ for any $k \in \mathbb{Z}_{\geq 0}$. 
By \eqref{A_tau_delta_def02} and Proposition \ref{chi epsilon} (ii) in \S 2, 
there exists a positive constant $C_{0}$ independent of $0 < \varepsilon < 1$ such that for any $x \in (0,\infty)$, 
\begin{align} 
|e^{i\lambda x^{p}} x^{q-1} a(x) \varphi(x) \chi_{\varepsilon}^{(k)}(x)| 
\leq C_{0} |a|^{(\tau)}_{0} x^{q-1} \langle x \rangle^{\tau} |\varphi(x)| =: M(x). 
\label{M_est} 
\end{align} 
Since $M$ is continuous on $(0,\infty)$ and $M \equiv 0$ on $[r,\infty)$, 
then $M$ is integrable on $[u,\infty)$ for any $u \in (0,\infty)$, and $M(x) = O(x^{\alpha})~(x \to +0)$ with $\alpha = q-1 > -1$. 
Thus $\int_{0}^{\infty} M(x) dx$ is absolutely convergent independent of $\chi_{\varepsilon}$. 
Therefore by Lebesgue's convergence theorem and Proposition \ref{chi epsilon} (i) and (iii) in \S 2, 
for each $k \in \mathbb{Z}_{\geq 0}$, 
there exists the following limit of improper integral independent of $\chi_{\varepsilon}$, and the following holds: 
\begin{align} 
\lim_{\varepsilon \to +0} \int_{0}^{\infty} e^{i\lambda x^{p}} x^{q-1} a(x) \varphi(x) \chi_{\varepsilon}^{(k)}(x) dx 
= \delta_{k0} \int_{0}^{\infty} e^{i\lambda x^{p}} x^{q-1} a(x) \varphi(x) dx. 
\notag 
\end{align} 

(ii) By Lemma \ref{L_star_l} (ii) when $h=1$ and $l=0$, 
for any $k \in \mathbb{Z}_{\geq 0}$, the following improper integral is absolutely convergent: 
\begin{align} 
\int_{0}^{\infty} e^{i\lambda x^{p}} x^{q-1} a(x) \psi(x) \chi_{\varepsilon}^{(k)}(x) dx. 
\label{tilde_I_phi_a_psi02} 
\end{align} 
In order to apply Lebesgue's convergence theorem, 
we shall show \eqref{tilde_I_phi_a_psi02} is bounded independent of $\chi_{\varepsilon}$ for any $k \in \mathbb{Z}_{\geq 0}$. 
By Lemma \ref{L_star_l} (iv) when $h=1$, 
for any $l \in \mathbb{N}$, 
\begin{align} 
\int_{0}^{\infty} e^{i\lambda x^{p}} x^{q-1} a(x) \psi(x) \chi_{\varepsilon}^{(k)}(x) dx 
= \int_{0}^{\infty} e^{i\lambda x^{p}} L^{*l} (x^{q-1} a(x) \psi(x) \chi_{\varepsilon}^{(k)}(x)) dx. 
\label{L star} 
\end{align} 
By Lemma \ref{L_star_l} (i) and Proposition \ref{chi epsilon} (ii) in \S 2, 
this means that the order of integrand descends to Lebesgue integrable by 
$L^{*l}$ for sufficiently large number $l \gg 0$. 
We shall show this. 
By \eqref{|f_{h,l}(x)|} when $h=1$, 
for each $l \in \mathbb{Z}_{\geq 0}$, there exist real constants $(C_{l,0},\dots,C_{l,l}) \ne (0,\dots,0)$ such that 
for any $x \in [0,\infty)$, 
\begin{align} 
|L^{*l} (x^{q-1} a(x) \psi(x) \chi_{\varepsilon}^{(k)}(x))| 
&\leq (\lambda p)^{-l} \sum_{j=0}^{l} |C_{l,j}| |x|^{q-1-pl+j} \notag \\ 
&\hspace{0.5cm}\times \sum_{s+t+u=j} \frac{j!}{s!t!u!} |a|^{(\tau)}_{l} \langle x \rangle^{\tau + \delta(s+t)} |\psi|^{(l)}_{r} |\chi_{\varepsilon}^{(k+u)}(x)|, 
\label{L^*l_est00} 
\end{align} 
where $|\psi|^{(l)}_{r} := \max_{t=0,\dots,l} \sup_{|x|< r} \langle x \rangle^{t} |\psi^{(t)}(x)|+1$. 
Here since $\psi \equiv 0$ for $|x| \leq 1$, $\mathrm{supp} \psi \cap (0,\infty) \subset [1,\infty)$. 
Hence if $x \in \mathrm{supp} \psi \cap (0,\infty) $, since $|x| \geq 1$, by \eqref{|x|_angle_x}, then $|x|^{\tau} \leq 2^{|\tau|/2} \langle x \rangle^{\tau}$ for $\tau \in \mathbb{R}$. 
And by Proposition \ref{chi epsilon} (ii) in \S 2 with $-1 \leq \delta$, for each $k \in \mathbb{Z}_{\geq 0}$ and for each $u=0,\dots,l$, 
there exists a positive constant $C_{k+u}$ independent of $0 < \varepsilon < 1$ such that for any $0 < \varepsilon < 1$ and for any $x \in [0,\infty)$, 
\begin{align} 
|\chi_{\varepsilon}^{(k+u)}(x)| \leq C_{k+u} \langle x \rangle^{-k-u} \leq C_{k+u} \langle x \rangle^{\delta u}. 
\label{partial_x_k+w_chi_varepsilon_x} 
\end{align} 
Hence by \eqref{L^*l_est00}, \eqref{|x|_angle_x} and \eqref{partial_x_k+w_chi_varepsilon_x}, 
for any $x \in [0,\infty)$, 
\begin{align} 
&|L^{*l} (x^{q-1} a(x) \psi(x) \chi_{\varepsilon}^{(k)}(x))| \notag \\ 
&\leq (\lambda p)^{-l} \sum_{j=0}^{l} |C_{l,j}| 2^{|q-1-pl+j|/2} \langle x \rangle^{q-1-pl+j} \notag \\ 
&\hspace{0.5cm}\times \sum_{s+t+u=j} \frac{j!}{s!t!u!} |a|^{(\tau)}_{l} \langle x \rangle^{\tau + \delta(s+t)} |\psi|^{(l)}_{r} C_{k+u} \langle x \rangle^{\delta u} \notag \\ 
&\leq (\lambda p)^{-l} \sum_{j=0}^{l} \max_{j=0,\dots,l} |C_{l,j}| 2^{|q-1-pl+l|/2} \notag \\ 
&\hspace{0.5cm}\times \sum_{s+t+u=j} \frac{j!}{s!t!u!} |a|^{(\tau)}_{l} |\psi|^{(l)}_{r} \max_{u=0,\dots,l} C_{k+u} \langle x \rangle^{q-1+\tau-pl+(1+\delta)l} \notag \\ 
&\leq (\lambda p)^{-l} \max_{j=0,\dots,l} |C_{l,j}| 2^{|q-1-pl+l|/2} \notag \\ 
&\hspace{0.5cm}\times \sum_{j=0}^{l} 3^{j} |a|^{(\tau)}_{l} |\psi|^{(l)}_{r} \max_{u=0,\dots,l} C_{k+u} \langle x \rangle^{(q+\tau)^{+}-1-(p-1-\delta)l} \notag \\ 
&= C^{(k)}_{l} \lambda^{-l} |a|^{(\tau)}_{l} \langle x \rangle^{\beta} =: M_{k}(x), 
\label{estimate of L star} 
\end{align} 
where 
\begin{align} 
C^{(k)}_{l} 
:&= p^{-l} \max_{j=0,\dots,l} |C_{l,j}| 2^{|q-1-pl+l|/2} \cdot \frac{3^{l+1}-1}{2} |\psi|^{(l)}_{r} \max_{u=0,\dots,l} C_{k+u}, 
\notag 
\end{align} 
and 
\begin{align} 
\beta = (q+\tau)^{+}-1 -(p-1-\delta)l. 
\label{beta_q-1+tau-(p-1-delta)m}
\end{align} 
Since $M_{k}$ is continuous on $[0,\infty)$, then $M_{k}$ is integrable on $[0,v]$ for any $v \in [0,\infty)$, and $M_{k}(x) = O(x^{\beta})~(x \to \infty)$. 
Moreover let
\begin{align} 
l_{p,q} := \left[ \frac{(q+\tau)^{+}}{p-1-\delta} \right] +1. 
\label{m_def} 
\end{align} 
Then by \eqref{beta_q-1+tau-(p-1-delta)m}, since $x-[x]-1<0$ for $x \in \mathbb{R}$, 
for any $l \in \mathbb{N}$ such that $l \geq l_{p,q}$, 
\begin{align} 
\beta 
&\leq (q+\tau)^{+}-1 -(p-1-\delta) l_{p,q} = (p-1-\delta) \left\{ \frac{(q+\tau)^{+}}{p-1-\delta} - l_{p,q} \right\} -1 \notag \\ 
&= (p-1-\delta) \left\{ \frac{(q+\tau)^{+}}{p-1-\delta} - \left[ \frac{(q+\tau)^{+}}{p-1-\delta} \right] -1 \right\} -1 < -1. 
\notag 
\end{align} 
Thus for any $k \in \mathbb{Z}_{\geq 0}$, $\int_{0}^{\infty} M_{k}(x) dx$ is absolutely convergent independent of $\chi_{\varepsilon}$. 

Therefore by applying Lebesgue's convergence theorem to the right hand side of \eqref{L star} as $\varepsilon \to +0$, and using Proposition \ref{chi epsilon} (i) and (iii) in \S 2, 
for each $k \in \mathbb{Z}_{\geq 0}$, 
there exists the following limit of improper integral independent of $\chi_{\varepsilon}$, and for any $l \in \mathbb{N}$ such that $l \geq l_{p,q}$, the following holds: 
\begin{align} 
&\lim_{\varepsilon \to +0} \int_{0}^{\infty} e^{i\lambda x^{p}} x^{q-1} a(x) \psi(x) \chi_{\varepsilon}^{(k)}(x) dx 
= \delta_{k0} \int_{0}^{\infty} e^{i\lambda x^{p}} L^{*l} (x^{q-1} a(x) \psi(x) ) dx. 
\notag 
\end{align} 

(iii) 
If $k \ne 0$, by (i) and (ii), since $\varphi + \psi \equiv 1$, then there exists the following limit of improper integral independent of $\chi_{\varepsilon}$, and the following holds: 
\begin{align} 
&\lim_{\varepsilon \to +0} \int_{0}^{\infty} e^{i\lambda x^{p}} x^{q-1} a(x) \chi_{\varepsilon}^{(k)}(x) dx 
= \lim_{\varepsilon \to +0} \int_{0}^{\infty} e^{i\lambda x^{p}} x^{q-1} a(x) \varphi(x) \chi_{\varepsilon}^{(k)}(x) dx \notag \\ 
&\hspace{0.25cm}+ \lim_{\varepsilon \to +0} \int_{0}^{\infty} e^{i\lambda x^{p}} x^{q-1} a(x) \psi(x) \chi_{\varepsilon}^{(k)}(x) dx = 0. 
\notag 
\end{align} 

(iv) 
If $k = 0$, by (i) and (ii), since $\varphi + \psi \equiv 1$, then there exists the following oscillatory integral, 
and for any $l \in \mathbb{N}$ such that $l \geq l_{p,q}$, the following holds: 
\begin{align} 
\tilde{I}_{p,q}^{+}[a](\lambda) 
:&= Os\text{-}\int_{0}^{\infty} e^{i\lambda x^{p}} x^{q-1} a(x) dx 
:= \lim_{\varepsilon \to +0} \int_{0}^{\infty} e^{i\lambda x^{p}} x^{q-1} a(x) \chi_{\varepsilon}(x) dx \notag \\ 
&= \int_{0}^{\infty} e^{i\lambda x^{p}} x^{q-1} a(x) \varphi(x) dx + \int_{0}^{\infty} e^{i\lambda x^{p}} L^{*l} (x^{q-1} a(x) \psi(x)) dx. 
\notag 
\end{align} 
Then by \eqref{M_est} and \eqref{estimate of L star} when $k=0$ and $\varepsilon \to +0$, for any $\lambda > 0$, 
\begin{align} 
|\tilde{I}_{p,q}^{+}[a](\lambda)| 
\leq C_{l} |a|^{(\tau)}_{l}, 
\label{tilde_I_m_est} 
\end{align} 
where $C_{l} = \int_{0}^{\infty} C_{0} x^{q-1} \langle x \rangle^{\tau} |\varphi(x)| dx + \int_{0}^{\infty} C^{(0)}_{l} \lambda^{-l} |\psi|^{(l)}_{r} \langle x \rangle^{\beta} dx$. 

(v) 
When $q > p$, 
let $l_{0} = [q/p)$. Since $(q/p)-1 \leq l_{0} < q/p$, then $0 < q-pl_{0} \leq p$. 
Since $q-pl_{0}+j >0$ and $a^{(j)} \in \mathcal{A}^{\tau +\delta j}_{\delta}(\mathbb{R})$ for $j=0,\dots,l_{0}$, 
by Lemma \ref{L_star_l} (iv) and (i) when $h=0$ and $k=0$, and by (iii) and (iv), 
\begin{align} 
\tilde{I}_{p,q}^{+}[a](\lambda) 
:&= \lim_{\varepsilon \to +0} \int_{0}^{\infty} e^{i\lambda x^{p}} x^{q-1} a(x) \chi_{\varepsilon}(x) dx \notag \\ 
&= \lim_{\varepsilon \to +0} \int_{0}^{\infty} e^{i\lambda x^{p}} L^{*l_{0}}(x^{q-1} a(x) \chi_{\varepsilon}(x)) dx \notag \\ 
&= \lim_{\varepsilon \to +0} \int_{0}^{\infty} e^{i\lambda x^{p}} \left( \frac{i}{\lambda p} \right)^{l_{0}} \sum_{j=0}^{l_{0}} C_{l_{0},j} x^{q-1-pl_{0}+j} (a(x) \chi_{\varepsilon}(x))^{(j)} dx \notag \\ 
&= \left( \frac{i}{\lambda p} \right)^{l_{0}} \sum_{j=0}^{l_{0}} C_{l_{0},j} \tilde{I}_{p,q-pl_{0}+j}^{+}[a^{(j)}](\lambda), 
\notag 
\end{align} 
where $(C_{l_{0},0},\dots, C_{l_{0},l_{0}}) \ne (0,\dots,0)$. 
Here for each $j=0,\dots,l_{0}$, let 
\begin{align} 
w_{j} := \left[ \frac{(q-pl_{0}+j+\tau +\delta j)^{+}}{p-1-\delta} \right] +1. 
\label{m_j_def} 
\end{align} 
Then by \eqref{tilde_I_m_est}, \eqref{m_def} and \eqref{A_tau_delta_def03}, 
there exists a positive constant $C_{w_{j}}$ such that for any $\lambda \geq 1$, 
\begin{align} 
|\tilde{I}_{p,q}^{+}[a](\lambda)| 
&\leq (\lambda p)^{-l_{0}} \sum_{j=0}^{l_{0}} |C_{l_{0},j}| |\tilde{I}_{p,q-pl_{0}+j}^{+}[a^{(j)}](\lambda)| 
\leq p^{-l_{0}} \sum_{j=0}^{l_{0}} |C_{l_{0},j}| C_{w_{j}} |a|^{(\tau)}_{j+w_{j}} \lambda^{-l_{0}} \notag \\ 
&\leq p^{-l_{0}} (l_{0}+1) \max_{j=0,\dots,l_{0}} |C_{l_{0},j}| C_{w_{j}} |a|^{(\tau)}_{l_{0}+w_{l_{0}}} \lambda^{-\frac{q}{p}+1} 
\leq C_{p,q} |a|^{(\tau)}_{l_{0}+w_{l_{0}}} \lambda^{-\frac{q-p}{p}}, 
\label{|tilde_I_p_q_+_a_(lambda)|} 
\end{align} 
where $C_{p,q} = p^{-l_{0}} (l_{0}+1) \max_{j=0,\dots,l_{0}} | C_{l_{0},j} | C_{w_{j}}$. 
Here by \eqref{m_j_def}, \eqref{m_def}, 
we see 
\begin{align} 
w_{l_{0}} 
&= \left[ \frac{\{ q+\tau - (p-1-\delta) l_{0} \}^{+}}{p-1-\delta} \right] +1 
\leq 
\left[ \frac{(q+\tau)^{+}}{p-1-\delta} \right] +1 =: 
l_{p,q}. 
\label{l_{0} + w_{l_{0}}} 
\end{align} 

Therefore by \eqref{|tilde_I_p_q_+_a_(lambda)|} and \eqref{l_{0} + w_{l_{0}}}, for any $\lambda \geq 1$, 
\begin{align} 
|\tilde{I}_{p,q}^{+}[a](\lambda)| 
\leq C_{p,q} |a|^{(\tau)}_{l_{0}+l_{p,q}} \lambda^{-\frac{q-p}{p}}. 
\notag 
\end{align} 
\end{proof} 

If $p = m \in \mathbb{N}$, $q=1$ and $k=0$, then the following holds: 
\begin{thm} 
\label{Os_m} 
Assume that $\lambda > 0$ and $m \in \mathbb{N}$. 
Let $a \in \mathcal{A}^{\tau}_{\delta}(\mathbb{R})$, 
$\varphi \in C^{\infty}_{0}(\mathbb{R})$ a cutoff function such that $\varphi \equiv 1$ on $|x| \leq 1$ and $\varphi \equiv 0$ on $|x| \geq r > 1$, 
$\psi := 1 - \varphi$, 
and let $\chi \in \mathcal{S}(\mathbb{R})$ with $\chi(0) = 1$, $0 < \varepsilon < 1$ and $\chi_{\varepsilon}(x) := \chi(\varepsilon x)$ for $x \in \mathbb{R}$, 
and let $L^{*} := - \frac{1}{i\lambda} \frac{d}{dx}  \frac{1}{mx^{m-1}}$ be the formal adjoint operator of $L := \frac{1}{mx^{m-1}} \frac{1}{i\lambda} \frac{d}{dx}$,
and 
\begin{align} 
l_{m,1} := \left[ \frac{(1+\tau)^{+}}{m-1-\delta} \right] +1. 
\notag 
\end{align} 
Then the following hold: 
\begin{enumerate} 
\item[(i)] 
There exist the following oscillatory integrals, and the following holds: 
\begin{align} 
&\tilde{J}_{m}^{\pm}[a \varphi](\lambda) 
:= Os\text{-}\int_{-\infty}^{\infty} e^{\pm i\lambda x^{m}} a(x) \varphi(x) dx 
= \int_{-\infty}^{\infty} e^{\pm i\lambda x^{m}} a(x) \varphi(x) dx. 
\notag 
\end{align} 
\item[(ii)] 
There exist the following oscillatory integrals, and for any $l \in \mathbb{N}$ such that $l \geq l_{m,1}$, the following holds: 
\begin{align} 
\tilde{J}_{m}^{\pm}[a \psi](\lambda) 
:= Os\text{-}\int_{-\infty}^{\infty} e^{\pm i\lambda x^{m}} a(x) \psi(x) dx 
= \int_{-\infty}^{\infty} e^{\pm i\lambda x^{m}} (\pm L^{*})^{l} (a(x) \psi(x)) dx. 
\notag 
\end{align} 
Then for each $l \in \mathbb{N}$ such that $l \geq l_{m,1}$, there exists a positive constant $C_{l}$ such that for any $\lambda > 0$, 
\begin{align} 
| \tilde{J}_{m}^{\pm}[a \psi(x)](\lambda) | \leq C_{l} \lambda^{-l} |a|^{(\tau)}_{l}, 
\notag 
\end{align} 
where $|a|^{(\tau)}_{l} := \max_{k=0,\dots, l} \sup_{x \in \mathbb{R}} \langle x \rangle^{-\tau - \delta k} |a^{(k)}(x)|$. 
\item[(iii)] 
There exist the following oscillatory integrals, and for any $l \in \mathbb{N}$ such that $l \geq l_{m,1}$, the following holds: 
\begin{align} 
\tilde{J}_{m}^{\pm}[a](\lambda) 
:&= Os\text{-}\int_{-\infty}^{\infty} e^{\pm i\lambda x^{m}} a(x) dx \notag \\ 
&= \int_{-\infty}^{\infty} e^{\pm i\lambda x^{m}} a(x) \varphi(x) dx + \int_{-\infty}^{\infty} e^{\pm i\lambda x^{m}} (\pm L^{*})^{l} (a(x) \psi(x)) dx, 
\notag 
\end{align} 
\end{enumerate} 
where double signs $\pm$ are in same order. 
\end{thm} 

\begin{proof} 
Since the lower side of double signs $\pm$ can be obtained as the conjugate of the upper one, 
we shall show the upper one. 

(i) 
By Theorem \ref{Lax02} (i) when $q=1$ and $k=0$, there exist the following oscillatory integrals, and the following holds: 
\begin{align} 
\tilde{I}_{m,1}^{\pm}[a \varphi](\lambda) 
:= Os\text{-}\int_{0}^{\infty} e^{\pm i\lambda x^{m}} a(x) \varphi(x) dx 
= \int_{0}^{\infty} e^{\pm i\lambda x^{m}} a(x) \varphi(x) dx, 
\label{tilde_I_m_1_pm_a_varphi_lambda)} 
\end{align} 
where double signs $\pm$ are in same order. 

By change of variable $x = -y$, 
since $a(-y) \in \mathcal{A}^{\tau}_{\delta}(\mathbb{R})$, $\varphi(-y) \in C^{\infty}_{0}(\mathbb{R})$ such that $\varphi(-y) \equiv 1$ on $|y| \leq 1$ and $\varphi(-y) \equiv 0$ on $|y| \geq r$ and $\chi(-y) \in \mathcal{S}(\mathbb{R})$ with $\chi(0) = 1$, then by \eqref{tilde_I_m_1_pm_a_varphi_lambda)}, 
\begin{align} 
&Os\text{-}\int_{-\infty}^{0} e^{i\lambda x^{m}} a(x) \varphi(x) dx 
:= \lim_{\varepsilon \to +0} \lim_{\substack{u \to +0\\v \to \infty}} \int_{-v}^{-u} e^{i\lambda x^{m}} a(x) \varphi(x) \chi_{\varepsilon}(x) dx \notag \\ 
&= \lim_{\varepsilon \to +0} \lim_{\substack{u \to +0\\v \to \infty}} \int_{v}^{u} e^{(-1)^{m}i\lambda y^{m}} a(-y) \varphi(-y) \chi_{\varepsilon}(-y) (-dy) \notag \\ 
&= Os\text{-}\int_{0}^{\infty} e^{(-1)^{m} i\lambda y^{m}} a(-y) \varphi(-y) dy 
= \int_{0}^{\infty} e^{(-1)^{m} i\lambda y^{m}} a(-y) \varphi(-y) dy 
\label{Os_m_varphi_change_of_variable} \\ 
&= \lim_{\substack{u \to +0\\v \to \infty}} \int_{u}^{v} e^{(-1)^{m}i\lambda y^{m}} a(-y) \varphi(-y) dy 
= \lim_{\substack{u \to +0\\v \to \infty}} \int_{-v}^{-u} e^{i\lambda x^{m}} a(x) \varphi(x) dx \notag \\ 
&= \int_{-\infty}^{0} e^{i\lambda x^{m}} a(x) \varphi(x) dx. 
\label{Os_int_-infty_0_e_ilambda_x_m_a(x)_varphi(x)_dx} 
\end{align} 
Hence by \eqref{tilde_I_m_1_pm_a_varphi_lambda)} and \eqref{Os_int_-infty_0_e_ilambda_x_m_a(x)_varphi(x)_dx}, 
\begin{align} 
&\tilde{J}_{m}^{+}[a \varphi](\lambda) 
:= Os\text{-}\int_{-\infty}^{\infty} e^{i\lambda x^{m}} a(x) \varphi(x) dx 
= \int_{-\infty}^{\infty} e^{i\lambda x^{m}} a(x) \varphi(x) dx. 
\notag 
\end{align} 

(ii) 
By Theorem \ref{Lax02} (ii) when $q=1$ and $k=0$, 
there exist the following oscillatory integrals, and for any $l \in \mathbb{N}$ such that $l \geq l_{m,1}$, the following holds: 
\begin{align} 
\tilde{I}_{m,1}^{\pm}[a \psi](\lambda) 
:= Os\text{-}\int_{0}^{\infty} e^{\pm i\lambda x^{m}} a(x) \psi(x) dx 
= \int_{0}^{\infty} e^{\pm i\lambda x^{m}} (\pm L^{*})^{l} (a(x) \psi(x)) dx, 
\label{tilde_I_m_1_pm_a_psi_lambda} 
\end{align} 
where $\pm$ are in same order. 
And by \eqref{estimate of L star} when $k=0$ and $\varepsilon \to +0$, for each $l \in \mathbb{N}$ such that $l \geq l_{m,1}$, there exists a positive constant $C_{l}$ such that for any $\lambda > 0$, 
\begin{align} 
| \tilde{I}_{m,1}^{\pm}[a \psi](\lambda) | 
\leq 2^{-1} C_{l} \lambda^{-l} |a|^{(\tau)}_{l}. 
\label{tilde_I_m_1_a_psi_est01} 
\end{align} 

By change of variable $x = -y$, then $a(-y) \in \mathcal{A}^{\tau}_{\delta}(\mathbb{R})$, $\varphi(-y) \in C^{\infty}_{0}(\mathbb{R})$ such that $\varphi(-y) \equiv 1$ on $|y| \leq 1$, 
$\varphi(-y) \equiv 0$ on $|y| \geq r$, 
$\chi(-y) \in \mathcal{S}(\mathbb{R})$ with $\chi(0) = 1$, 
and 
\begin{align} 
L_{x}^{*} 
&= - \frac{1}{i\lambda} \frac{d}{dx}  \frac{1}{mx^{m-1}} 
= - \frac{1}{i\lambda} \frac{1}{\frac{dx}{dy}} \frac{d}{dy}  \frac{1}{m (-y)^{m-1}} 
= (-1)^{m} L_{y}^{*}. 
\notag 
\end{align} 
Then by \eqref{tilde_I_m_1_pm_a_psi_lambda}, for any $l \in \mathbb{N}$ such that $l \geq l_{m,1}$, 
\begin{align} 
&Os\text{-}\int_{-\infty}^{0} e^{i\lambda x^{m}} a(x) \psi(x) dx 
:= \lim_{\varepsilon \to +0} \lim_{\substack{u \to +0\\v \to \infty}} \int_{-v}^{-u} e^{i\lambda x^{m}} a(x) \psi(x) \chi_{\varepsilon}(x) dx \notag \\ 
&= \lim_{\varepsilon \to +0} \lim_{\substack{u \to +0\\v \to \infty}} \int_{v}^{u} e^{(-1)^{m}i\lambda y^{m}} a(-y) \psi(-y) \chi_{\varepsilon}(-y) (-dy) \notag \\ 
&= Os\text{-}\int_{0}^{\infty} e^{(-1)^{m} i\lambda y^{m}} a(-y) \psi(-y) dy 
\label{Os_m_psi_change_of_variable} \\ 
&= \int_{0}^{\infty} e^{(-1)^{m} i\lambda y^{m}} ((-1)^{m} L_{y}^{*})^{l} (a(-y) \psi(-y)) dy \notag \\ 
&= \lim_{\substack{u \to +0\\v \to \infty}} \int_{u}^{v} e^{(-1)^{m}i\lambda y^{m}} ((-1)^{m} L_{y}^{*})^{l} (a(-y) \psi(-y)) dy \notag \\ 
&= \lim_{\substack{u \to +0\\v \to \infty}} \int_{-v}^{-u} e^{i\lambda x^{m}} L_{x}^{*l}(a(x) \psi(x)) dx 
= \int_{-\infty}^{0} e^{i\lambda x^{m}} L_{x}^{*l}(a(x) \psi(x)) dx. 
\label{Os_int_-infty_0_e_ilambda_x_m_a(x)_psi(x)_dx} 
\end{align} 
And then by \eqref{tilde_I_m_1_a_psi_est01} and \eqref{Os_m_psi_change_of_variable}, for any $\lambda > 0$, the following holds: 
\begin{align} 
| \tilde{I}_{m,1}^{\pm_{m}}[a(-y) \psi(-y)](\lambda) | 
\leq 2^{-1} C_{l} \lambda^{-l} |a|^{(\tau)}_{l}. 
\label{tilde_I_m_1_a_psi_est02} 
\end{align} 

Hence by \eqref{tilde_I_m_1_pm_a_psi_lambda} and \eqref{Os_int_-infty_0_e_ilambda_x_m_a(x)_psi(x)_dx}, for any $l \in \mathbb{N}$ such that $l \geq l_{m,1}$, 
\begin{align} 
\tilde{J}_{m}^{+}[a \psi](\lambda) 
:&= Os\text{-}\int_{-\infty}^{\infty} e^{i\lambda x^{m}} a(x) \psi(x) dx 
= \int_{-\infty}^{\infty} e^{i\lambda x^{m}} L^{*l}(a(x) \psi(x)) dx. 
\notag 
\end{align} 
And by \eqref{tilde_I_m_1_a_psi_est01} and \eqref{tilde_I_m_1_a_psi_est02}, for each $l \in \mathbb{N}$ such that $l \geq l_{m,1}$, there exists a positive constant $C_{l}$ such that for any $\lambda > 0$, 
\begin{align} 
| \tilde{J}_{m}^{\pm}[a \psi(x)](\lambda) | \leq C_{l} \lambda^{-l} |a|^{(\tau)}_{l}. 
\notag 
\end{align} 

(iii) By (i) and (ii), since $\varphi + \psi \equiv 1$, then there exists the following oscillatory integral, and for any $l \in \mathbb{N}$ such that $l \geq l_{m,1}$, and the following holds: 
\begin{align} 
\tilde{J}_{m}^{+}[a](\lambda) 
:&= Os\text{-}\int_{-\infty}^{\infty} e^{i\lambda x^{m}} a(x) dx \notag \\ 
&= \int_{-\infty}^{\infty} e^{i\lambda x^{m}} a(x) \varphi(x) dx + \int_{-\infty}^{\infty} e^{i\lambda x^{m}} L^{*l} (a(x) \psi(x)) dx. 
\notag 
\end{align} 
\end{proof} 

\section{Generalized Fresnel Integrals} 

In this section, 
we consider a generalization of the Fresnel integrals. 

\begin{lem} 
\label{Generalized the Fresnel integrals} 
Assume that $p > q > 0$. 
Then the following holds: 
\begin{align} 
I_{p,q}^{\pm} 
:= \int_{0}^{\infty} e^{\pm ix^{p}} x^{q-1} dx 
= p^{-1} e^{\pm i\frac{\pi}{2} \frac{q}{p}} \varGamma \left( \frac{q}{p} \right), 
\label{I_pq} 
\end{align} 
where $\varGamma$ is the Gamma function 
and double signs $\pm$ are in same order. 
\end{lem} 

\begin{proof} 
Since the lower side of double signs $\pm$ can be obtained as the conjugate of the upper one, 
we shall show the upper one. 
Suppose $p > q > 0$. 
Consider the following: 
\begin{align} 
C_{1} &:= \{ z=r \in \mathbb{C} | 0 < \varepsilon  \leq r \leq R \} , \notag \\ 
C_{2} &:= \{ z=R e^{i\theta} \in \mathbb{C} | 0 \leq \theta \leq \pi /2p \} , \notag \\ 
C_{3} &:= \{ z=-se^{i(\pi /2p)} \in \mathbb{C} | -R \leq s \leq -\varepsilon \} , \notag \\ 
C_{4} &:= \{ z=\varepsilon  e^{-i\tau} \in \mathbb{C} | -\pi /2p \leq \tau \leq 0 \} \notag 
\end{align} 
and a domain $D$ with the anticlockwise oriented boundary $\sum_{j=1}^{4} C_{j}$. 
Since $e^{iz^{p}} z^{q-1}$ is holomorphic in $D$ for $p > q > 0$, 
by Cauchy's integral theorem, 
\begin{align} 
0 
= \int_{\sum_{j=1}^{4} C_{j}} e^{iz^{p}} z^{q-1} dz 
= \sum_{j=1}^{4} \int_{C_{j}} e^{iz^{p}} z^{q-1} dz. 
\label{Cauchy} 
\end{align} 

As to $\int_{C_{2}} e^{iz^{p}} z^{q-1} dz$, 
by Jordan's inequality: $2/\pi < (\sin x)/x$ for $0 < x < \pi/2$, 
since $(2/\pi) p\theta < \sin (p\theta)$ for $0 < \theta < \pi/2p$, as $R \to \infty$, 
\begin{align} 
&\left| \int_{C_{2}} e^{iz^{p}} z^{q-1} dz \right| 
= \left| \int_{0}^{\frac{\pi}{2p}} e^{i(Re^{i\theta})^{p}} (Re^{i\theta})^{q-1} Rie^{i\theta} d\theta \right| 
\leq R^{q} \int_{0}^{\frac{\pi}{2p}} e^{-R^{p} \sin (p\theta)} d\theta \notag \\ 
&< R^{q} \int_{0}^{\frac{\pi}{2p}} e^{-R^{p} \frac{2}{\pi} p\theta} d\theta 
= R^{q} \left[ -R^{-p} \frac{\pi}{2p} e^{-R^{p} \frac{2}{\pi} p\theta} \right]_{0}^{\frac{\pi}{2p}} 
= \frac{\pi}{2p} \frac{1 - e^{-R^{p}}}{R^{p-q}} \to 0. 
\label{C_2_integral} 
\end{align} 

As to $\int_{C_{4}} e^{iz^{p}} z^{q-1} dz$, 
by change of variable $\tau = -\theta$, 
as $\varepsilon \to +0$, we similarly have 
\begin{align} 
&\left| \int_{C_{4}} e^{iz^{p}} z^{q-1} dz \right| 
= \left| \int_{-\frac{\pi}{2p}}^{0} e^{i(\varepsilon e^{-i\tau})^{p}} (\varepsilon e^{-i\tau})^{q-1} (-\varepsilon ie^{-i\tau}) d\tau \right| \notag \\ 
&= \left| \int_{\frac{\pi}{2p}}^{0} e^{i(\varepsilon e^{i\theta})^{p}} (\varepsilon e^{i\theta})^{q-1} (-\varepsilon ie^{i\theta}) (-d\theta) \right| 
\leq {\varepsilon }^{q} \int_{0}^{\frac{\pi}{2p}} e^{-{\varepsilon }^{p} \sin (p\theta)} d\theta \notag \\ 
&< {\varepsilon }^{q} \int_{0}^{\frac{\pi}{2p}} e^{-{\varepsilon }^{p} \frac{2}{\pi} p\theta} d\theta 
= {\varepsilon }^{q} \left[ -{\varepsilon }^{-p} \frac{\pi}{2p} e^{-{\varepsilon }^{p} \frac{2}{\pi} p\theta} \right]_{0}^{\frac{\pi}{2p}} 
= \frac{\pi}{2p} \frac{1 - e^{-\varepsilon^{p}}}{\varepsilon^{p-q}} \to 0. 
\label{C_4_integral} 
\end{align} 
Here we used L'H\^{o}pital's rule as follows, 
\begin{align} 
\lim_{\varepsilon \to +0} \frac{1 - e^{-\varepsilon^{p}}}{\varepsilon^{p-q}} 
= \lim_{\varepsilon \to +0} \frac{p\varepsilon^{p-1} e^{-\varepsilon^{p}}}{(p-q) \varepsilon^{p-q-1}} 
= \lim_{\varepsilon \to +0} \frac{p\varepsilon^{q} e^{-\varepsilon^{p}}}{p-q} 
= 0. 
\notag 
\end{align} 

Next put $f(x) = e^{ix^{p}} x^{q-1}$ for $x \in (0,\infty)$. 
Since $f$ is continuous on $(0,\infty)$, 
then $f$ is integrable on $[u,1]$ for any $u \in (0,1]$, and $f(x) = O(x^{\alpha})~(x \to +0)$ with $\alpha = q-1>-1$. 
Thus $\int_{0}^{1} e^{ix^{p}} x^{q-1} dx$ is absolutely convergent. 
And $f$ is also integrable on $[1,v]$ for any $v \in [1,\infty)$. 
Using $L := \frac{1}{px^{p-1}} \frac{1}{i} \frac{d}{dx}$, 
since $L(e^{ix^{p}}) = e^{ix^{p}}$ when $x \ne 0$, 
by integration by parts, 
\begin{align} 
&\int_{1}^{v} e^{ix^{p}} x^{q-1} dx 
= \int_{1}^{v} L(e^{ix^{p}}) x^{q-1} dx 
= \int_{1}^{v} \frac{1}{px^{p-1}} \frac{1}{i} \frac{d}{dx} (e^{ix^{p}}) x^{q-1} dx \notag \\ 
&= \int_{1}^{v} \frac{1}{ip} \frac{d}{dx} (e^{ix^{p}}) x^{q-p} dx 
= \frac{1}{ip} \left\{ \Big[ e^{ix^{p}} x^{q-p} \Big] _{1}^{v} - (q-p) \int_{1}^{v} e^{ix^{p}} x^{q-p-1} dx \right\}. 
\notag 
\end{align} 
Here $|e^{ix^{p}} x^{q-p}| \to 0$ as $x \to \infty$. 
Put $g(x) = e^{ix^{p}} x^{q-p-1} $. Since $g$ is continuous on $[1,\infty)$, 
then $g$ is integrable on $[1,v]$ for any $v \in [1,\infty)$, and $g(x) = O(x^{\beta})~(x \to \infty)$ with $\beta = q-p-1 < -1$. 
Thus $\int_{1}^{\infty} e^{ix^{p}} x^{q-p-1} dx$ is absolutely convergent. 
Hence 
\begin{align} 
\int_{1}^{\infty} e^{ix^{p}} x^{q-1} dx 
&= \frac{1}{ip} \left\{ -e^{i} - (q-p) \int_{1}^{\infty} e^{ix^{p}} x^{q-p-1} dx \right\} 
\label{int_0_infty_e_ixp_x_q-1} 
\end{align} 
is also absolutely convergent. 
Hence 
\begin{align} 
I_{p,q}^{+} 
:&= \int_{0}^{\infty} e^{ix^{p}} x^{q-1} dx 
= \int_{0}^{1} e^{ix^{p}} x^{q-1} dx + \int_{1}^{\infty} e^{ix^{p}} x^{q-1} dx 
\notag 
\end{align} 
are absolutely convergent. 
Therefore by \eqref{Cauchy}, \eqref{C_2_integral} and \eqref{C_4_integral}, 
\begin{align} 
I_{p,q}^{+} 
:&= \int_{0}^{\infty} e^{ix^{p}} x^{q-1} dx 
= \lim_{\substack{\varepsilon \to +0\\ R \to \infty}} \int_{\varepsilon}^{R} e^{ir^{p}} x^{r-1} dr 
= \lim_{\substack{\varepsilon \to +0\\ R \to \infty}} \int_{C_{1}} e^{iz^{p}} z^{q-1} dz \notag \\ 
&= -\lim_{\substack{\varepsilon \to +0\\ R \to \infty}} \int_{C_{3}} e^{iz^{p}} z^{q-1} dz 
= - \lim_{\substack{\varepsilon \to +0\\ R \to \infty}} \int_{-R}^{-\varepsilon} e^{i(-se^{i\frac{\pi}{2p}})^{p}} (-se^{i\frac{\pi}{2p}})^{q-1} (-e^{i\frac{\pi}{2p}} ds). \notag 
\end{align} 
Moreover, 
by change of variables $s = -r$ and $r=t^{1/p}$, 
\begin{align} 
&I_{p,q}^{+} 
= -\lim_{\substack{\varepsilon \to +0\\ R \to \infty}} \int_{R}^{\varepsilon} e^{i(re^{i\frac{\pi}{2p}})^{p}} (re^{i\frac{\pi}{2p}})^{q-1} (-e^{i\frac{\pi}{2p}}) (-dr) 
= e^{i\frac{\pi}{2} \frac{q}{p}} \lim_{\substack{\varepsilon \to +0\\ R \to \infty}} \int_{\varepsilon}^{R} e^{-r^{p}} r^{q-1} dr \notag \\ 
&= e^{i\frac{\pi}{2} \frac{q}{p}} \lim_{\substack{\varepsilon \to +0\\ R \to \infty}} \int_{\varepsilon^{p}}^{R^{p}} e^{-t} t^{\frac{q-1}{p}} p^{-1} t^{\frac{1}{p}-1} dt 
= p^{-1} e^{i\frac{\pi}{2} \frac{q}{p}} \int_{0}^{\infty} e^{-t} t^{\frac{q}{p}-1} dt 
= p^{-1} e^{i\frac{\pi}{2} \frac{q}{p}} \varGamma \left( \frac{q}{p} \right). \notag 
\end{align} 
\end{proof} 

When $q \geq p > 0$, 
we can make a sense of \eqref{I_pq} as oscillatory integrals. 
By Lemma \ref{L_star_l}, Theorem \ref{Lax02} in $\S 3$ and Lemma \ref{Generalized the Fresnel integrals}, 
we obtain the following theorem: 
\begin{thm}
\label{th01}
Assume that $p,q \in \mathbb{C}$. 
\begin{enumerate} 
\item[(i)] 
If $p > 0$ and $q > 0$, 
then 
\begin{align} 
\tilde{I}_{p,q}^{\pm} 
:= Os\mbox{-}\int_{0}^{\infty} e^{\pm ix^{p}} x^{q-1} dx 
= p^{-1} e^{\pm i\frac{\pi}{2} \frac{q}{p}} \varGamma \left( \frac{q}{p} \right). 
\label{generalized_Fresnel_integral_def} 
\end{align} 
\item[(ii)] 
The $\tilde{I}_{p,q}^{\pm}$ can be extended non-zero meromorphic on $\mathbb{C}$ with poles of order 1 at $q = -pj$ for $j \in \mathbb{N}$ as to $q$ for each $p \in \mathbb{C} \setminus \{ 0 \}$, 
and meromorphic on $\mathbb{C} \setminus \{ 0 \}$ with poles of order 1 at $p = -q/j$ for $j \in \mathbb{N}$ as to $p$ for each $q \in \mathbb{C}$ by analytic continuation, 
\end{enumerate} 
where double signs $\pm$ are in same order. 
We call $\tilde{I}^{\pm}_{p,q}$ ``generalized Fresnel integrals''. 
\end{thm} 

\begin{proof} 
Since the lower side of double signs $\pm$ can be obtained as the conjugate of the upper one, 
we shall show the upper one. 
Since $a \equiv 1 \in \mathcal{A}^{0}_{-1}(\mathbb{R})$, 
we can use Theorem \ref{Lax02} in $\S 3$ when $\lambda = 1$, $k=0$ and $a \equiv 1$. 

(i) 
Suppose $p > 0$ and $q > 0$. 
Let $\chi \in \mathcal{S}(\mathbb{R})$ with $\chi(0) = 1$, $0 < \varepsilon < 1$ and $\chi_{\varepsilon}(x) := \chi(\varepsilon x)$ for $x \in \mathbb{R}$. 
By Theorem \ref{Lax02} (iv) in $\S 3$, there exists the following oscillatory integral: 
\begin{align} 
\tilde{I}_{p,q}^{+} 
:= \tilde{I}_{p,q}^{+}[1](1) 
:= Os\text{-}\int_{0}^{\infty} e^{ix^{p}} x^{q-1} dx 
:= \lim_{\varepsilon \to +0} \int_{0}^{\infty} e^{ix^{p}} x^{q-1} \chi_{\varepsilon}(x) dx. 
\notag 
\end{align} 

When $p > q$, put $f(x) = e^{ix^{p}} x^{q-1} \chi_{\varepsilon}(x)$. 
Since $f$ is continuous on $(0,\infty)$, then $f$ is integrable on $[u,1]$ for any $u \in (0,1]$, 
and $f(x) = O(x^{\alpha})~(x \to +0)$ with $\alpha = q-1 > -1$. 
Thus $\int_{0}^{1} e^{ix^{p}} x^{q-1} \chi_{\varepsilon}(x) dx$ is absolutely convergent. 
By Proposition \ref{chi epsilon} (ii) in \S 2, there exists a positive constant $C_{0}$ independent of $0 < \varepsilon < 1$ such that for any $x \in (0,1]$, 
\begin{align} 
| f(x) | 
= | e^{ix^{p}} x^{q-1} \chi_{\varepsilon}(x) | 
\leq C_{0} |x|^{q-1} =: M(x). 
\notag 
\end{align} 
Since $M$ is continuous on $(0,1]$, then $M$ is integrable on $[u,1]$ for any $u \in (0,1]$, 
and $M(x) = O(x^{\alpha})~(x \to 0)$ with $\alpha = q-1 > -1$. 
Thus $\int_{0}^{1} M(x) dx$ is absolutely convergent independent of $\chi_{\varepsilon}$. 
Hence by Lebesgue's convergence theorem and Proposition \ref{chi epsilon} (i) in \S 2, 
there exists the following oscillatory integral, and the following holds: 
\begin{align} 
Os\text{-}\int_{0}^{1} e^{ix^{p}} x^{q-1} dx 
:= \lim_{\varepsilon \to +0} \int_{0}^{1} e^{ix^{p}} x^{q-1} \chi_{\varepsilon}(x) dx 
= \int_{0}^{1} e^{ix^{p}} x^{q-1} dx. 
\label{Os_0_1} 
\end{align} 
And $f$ is also integrable on $[1,v]$ for any $v \in [1,\infty)$, 
and $f(x) = O(x^{\beta})~(x \to \infty)$ with $\beta = q-1-m$ for any $m \in \mathbb{Z}_{\geq 0}$. 
Here let $m = [q]+1$. 
Since $x-[x]-1<0$ for $x \in \mathbb{R}$, then $\beta < -1$. 
Thus $\int_{1}^{\infty} e^{ix^{p}} x^{q-1} \chi_{\varepsilon}(x) dx$ is absolutely convergent. 
By integration by parts, 
\begin{align} 
&\int_{1}^{\infty} e^{ix^{p}} x^{q-1} \chi_{\varepsilon}(x) dx 
= \lim_{v \to \infty} \int_{1}^{v} e^{ix^{p}} x^{q-1} \chi_{\varepsilon}(x) dx \notag \\ 
&= \lim_{v \to \infty} \int_{1}^{v} \frac{1}{px^{p-1}} \frac{1}{i} \frac{d}{dx} (e^{ix^{p}}) x^{q-1} \chi_{\varepsilon}(x) dx \notag \\ 
&= \lim_{v \to \infty} \int_{1}^{v} \frac{1}{ip} \frac{d}{dx} (e^{ix^{p}}) x^{q-p} \chi_{\varepsilon}(x) dx 
= \lim_{v \to \infty} \frac{1}{ip} \bigg\{ \Big[ e^{ix^{p}} x^{q-p} \chi_{\varepsilon}(x) \Big] _{1}^{v} \bigg. \notag \\ 
&\bigg. \hspace{0.5cm}- (q-p) \int_{1}^{v} e^{ix^{p}} x^{q-p-1} \chi_{\varepsilon}(x) dx - \int_{1}^{v} e^{ix^{p}} x^{q-p} \chi'_{\varepsilon}(x) dx \bigg\}. 
\label{int_1_infty_e_ix_p_x_q-1_chi_varepsilon_x_dx_01} 
\end{align} 
Here $| e^{ix^{p}} x^{q-p} \chi_{\varepsilon}(x) | \to 0$ as $x \to \infty$. 
And put $g_{j}(x) = e^{ix^{p}} x^{q-p-1+j} \chi^{(j)}_{\varepsilon}(x)$ for $j=0,1$. 
Since $g_{j}$ is continuous on $[1,\infty)$, then $g_{j}$ is integrable on $[1,v]$ for any $v \in [1,\infty)$, 
and $g_{j}(x) = O(x^{\beta})~(x \to \infty)$ with $\beta = q-p-1+j-m$ for any $m \in \mathbb{Z}_{\geq 0}$. 
Here let $m = j$. Then $\beta < -1$. 
Thus $\int_{1}^{\infty} e^{ix^{p}} x^{q-p-1+j} \chi^{(j)}_{\varepsilon}(x) dx$ is absolutely convergent for $j=0,1$. 
Hence by \eqref{int_1_infty_e_ix_p_x_q-1_chi_varepsilon_x_dx_01}, 
\begin{align} 
&\int_{1}^{\infty} e^{ix^{p}} x^{q-1} \chi_{\varepsilon}(x) dx \notag \\ 
&= \frac{1}{ip} \left\{ -e^{i}\chi_{\varepsilon}(1) - (q-p) \int_{1}^{\infty} e^{ix^{p}} x^{q-p-1} \chi_{\varepsilon}(x) dx - \int_{1}^{\infty} e^{ix^{p}} x^{q-p} \chi'_{\varepsilon}(x) dx \right\}. 
\label{int_1_infty_e_ix_p_x_q-1_chi_varepsilon_x_dx_02} 
\end{align} 
Here noting $|x| = (|x|^{2})^{1/2} \leq \langle x \rangle$, 
by Proposition \ref{chi epsilon} (ii) in \S 2, 
for each $j=0,1$, there exists a positive constant $C_{j}$ independent of $0 < \varepsilon < 1$ such that for any $x \in [1,\infty)$, 
\begin{align} 
| g_{j}(x) | 
= | e^{ix^{p}} x^{q-p-1+j} \chi^{(j)}_{\varepsilon}(x) | 
\leq |x|^{q-p-1+j} C_{j} \langle x \rangle^{-j} 
\leq C_{j} |x|^{q-p-1} =: M_{j}(x). 
\notag 
\end{align} 
Since $M_{j}$ is continuous on $[1,\infty)$, then $M_{j}$ is integrable on $[1,v]$ for any $v \in [1,\infty)$, 
and $M_{j}(x) = O(x^{\beta})~(x \to \infty)$ with $\beta = q-p-1 < -1$. 
Thus $\int_{1}^{\infty} M_{j}(x) dx$ is absolutely convergent independent of $\chi_{\varepsilon}$ for $j=0,1$. 
Hence by \eqref{int_1_infty_e_ix_p_x_q-1_chi_varepsilon_x_dx_02}, Lebesgue's convergence theorem, and Proposition \ref{chi epsilon} (i) and (iii) in \S 2, there exists the following oscillatory integral, 
and by \eqref{int_0_infty_e_ixp_x_q-1}, the following holds: 
\begin{align} 
&Os\text{-}\int_{1}^{\infty} e^{ix^{p}} x^{q-1} dx 
:= \lim_{\varepsilon \to +0} \int_{1}^{\infty} e^{ix^{p}} x^{q-1} \chi_{\varepsilon}(x) dx \notag \\ 
&= \frac{1}{ip} \left\{ -e^{i} - (q-p) \int_{1}^{\infty} e^{ix^{p}} x^{q-p-1} dx \right\} 
= \int_{1}^{\infty} e^{ix^{p}} x^{q-1} dx. 
\label{Os_1_infty} 
\end{align} 
Therefore by \eqref{Os_0_1}, \eqref{Os_1_infty} and \eqref{I_pq}, 
\begin{align} 
I_{p,q}^{+} 
:= Os\text{-}\int_{0}^{\infty} e^{ix^{p}} x^{q-1} dx 
= \int_{0}^{\infty} e^{ix^{p}} x^{q-1} dx 
= p^{-1} e^{i\frac{\pi}{2} \frac{q}{p}} \varGamma \left( \frac{q}{p} \right). 
\label{Generalized Fresnel Oscillatory integral p > q > 0} 
\end{align} 

When $q=p$, 
by integration by parts, Lemma \ref{L_star_l} (ii), and Theorem \ref{Lax02} (iii) in $\S 3$, 
\begin{align} 
\tilde{I}_{p,p}^{+} 
&= \lim_{\varepsilon \to +0} \lim_{\substack{u \to +0\\ v \to \infty}} \int_{u}^{v} e^{ix^{p}} x^{p-1} \chi_{\varepsilon}(x) dx 
= \lim_{\varepsilon \to +0} \lim_{\substack{u \to +0\\ v \to \infty}} \int_{u}^{v} \frac{1}{ip} \frac{d}{dx} (e^{ix^{p}}) \chi_{\varepsilon}(x) dx \notag \\ 
&= \lim_{\varepsilon \to +0} \lim_{\substack{u \to +0\\ v \to \infty}} \frac{1}{ip} \left( \Big[ e^{ix^{p}} \chi_{\varepsilon}(x) \Big]_{u}^{v} 
- \int_{u}^{v} e^{ix^{p}} \chi'_{\varepsilon}(x) dx \right) \notag \\ 
&= \lim_{\varepsilon \to +0} \frac{1}{ip} \left( -1 - \int_{0}^{\infty} e^{ix^{p}} \chi'_{\varepsilon}(x) dx \right) 
= \frac{i}{p} 
= p^{-1} e^{i\frac{\pi}{2} \frac{p}{p}} \varGamma \left( \frac{p}{p} \right). 
\label{I_pp} 
\end{align} 

When $q>p$, 
let $l_{0} = [q/p)$. Since $(q/p)-1 \leq l_{0} < q/p$, then $0 <q-pl_{0} \leq p$. 
By Lemma \ref{L_star_l} (iv) and (i), and Theorem \ref{Lax02} (iii) in $\S 3$, 
\begin{align} 
\tilde{I}_{p,q}^{+} 
&= \lim_{\varepsilon \to +0} \int_{0}^{\infty} e^{ix^{p}} x^{q-1} \chi_{\varepsilon}(x) dx 
= \lim_{\varepsilon \to +0} \int_{0}^{\infty} e^{ix^{p}} L^{*l_{0}}(x^{q-1} \chi_{\varepsilon}(x)) dx \notag \\ 
&= \lim_{\varepsilon \to +0} \int_{0}^{\infty} e^{ix^{p}} \left( \frac{i}{p} \right)^{l_{0}} \sum_{j=0}^{l_{0}} C_{l_{0},j} x^{q-1-pl_{0}+j} \chi_{\varepsilon}^{(j)}(x) dx \notag \\ 
&= \left( \frac{i}{p} \right)^{l_{0}} C_{l_{0},0} \lim_{\varepsilon \to +0} \int_{0}^{\infty} e^{ix^{p}} x^{q-pl_{0}-1} \chi_{\varepsilon}(x) dx 
= \left( \frac{i}{p} \right)^{l_{0}} \prod_{s=1}^{l_{0}} (q-ps) \tilde{I}_{p,q-pl_{0}}^{+}. 
\label{I_pq02} 
\end{align} 
When $q-pl_{0} = p$, that is, $q=p(l_{0}+1)$, then by \eqref{I_pq02} and \eqref{I_pp}, 
\begin{align} 
\tilde{I}_{p,p(l_{0}+1)}^{+} 
&= \left( \frac{i}{p} \right)^{l_{0}} \prod_{s=1}^{l_{0}} \{ p(l_{0}+1)-ps \} \tilde{I}_{p,p}^{+} 
= i^{l_{0}} l_{0}! \frac{i}{p} 
= p^{-1} i^{l_{0}+1} l_{0}! \notag \\ 
&= p^{-1} e^{i\frac{\pi}{2} (l_{0}+1)} \varGamma (l_{0}+1) 
= p^{-1} e^{i\frac{\pi}{2} \frac{q}{p}} \varGamma \left( \frac{q}{p} \right). 
\notag 
\end{align} 
When $q-pl_{0} < p$, then 
by \eqref{I_pq02} and \eqref{Generalized Fresnel Oscillatory integral p > q > 0}, 
\begin{align} 
\tilde{I}_{p,q}^{+} 
&= \left( \frac{i}{p} \right)^{l_{0}} \prod_{s=1}^{l_{0}} (q-ps) I_{p,q-pl_{0}}^{+}  
= i^{l_{0}} \prod_{s=1}^{l_{0}} \left( \frac{q-ps}{p} \right) p^{-1} e^{i\frac{\pi}{2} \frac{q-pl_{0}}{p}} \varGamma \left( \frac{q-pl_{0}}{p} \right) \notag \\ 
&= p^{-1} e^{i\frac{\pi}{2} l_{0}} e^{i\frac{\pi}{2} \left( \frac{q}{p} - l_{0} \right)} \prod_{s=1}^{l_{0}} \left( \frac{q}{p}-s \right) \varGamma \left( \frac{q}{p} - l_{0} \right) 
= p^{-1} e^{i\frac{\pi}{2} \frac{q}{p}} \varGamma \left( \frac{q}{p} \right). 
\notag 
\end{align} 

(ii) 
Since $e^{\pm i\frac{\pi}{2} z}$ are non-zero holomorphic on $\mathbb{C}$, 
since $\varGamma (z)$ can be extended non-zero meromorphic on $\mathbb{C}$ with poles of order 1 at $z = -j$ for $j \in \mathbb{N}$ by analytic continuation, 
since $f(q) = q/p$ is holomorphic on $\mathbb{C}$ for each $p \in \mathbb{C} \setminus \{ 0 \}$, 
and since $g(p) = q/p$ is holomorphic on $\mathbb{C} \setminus \{ 0 \}$ for each $q \in \mathbb{C}$, 
then $\tilde{I}_{p,q}^{+} = p^{-1} e^{i\frac{\pi}{2} f(q)} \varGamma (f(q))$ can be extended non-zero meromorphic on $\mathbb{C}$ with poles of order 1 at $q = -pj$ for $j \in \mathbb{N}$ as to $q$ for each $p \in \mathbb{C} \setminus \{ 0 \}$, 
and $\tilde{I}_{p,q}^{+} = p^{-1} e^{i\frac{\pi}{2} g(p)} \varGamma (g(p))$ can be extended meromorphic on $\mathbb{C} \setminus \{ 0 \}$ with poles of order 1 at $p = -q/j$ for $j \in \mathbb{N}$ as to $p$ for each $q \in \mathbb{C}$. 
\end{proof} 

Using the theorem above, 
we can extend the Euler Beta function as follows. 
\begin{prop} 
Assume that $p_{j} > 0$ and $q_{j} 
\in \mathbb{C} \setminus \{ -p_{j}\mathbb{N} \}$ for $j=1,2,3$. 
Let 
\begin{align} 
\tilde{B}^{\pm}(p_{1},p_{2},p_{3};q_{1},q_{2},q_{3}) 
:= e^{\mp i \frac{\pi}{2} \left( \frac{q_{1}}{p_{1}} + 
\frac{q_{2}}{p_{2}} - \frac{q_{3}}{p_{3}} \right)} 
\frac{p_{1} p_{2}}{p_{3}} \frac{\tilde{I}_{p_{1},q_{1}}^{\pm} 
\tilde{I}_{p_{2},q_2}^{\pm}}{\tilde{I}_{p_{3},q_{3}}^{\pm}}. \notag 
\end{align} 
Then 
\begin{align} 
\tilde{B}^{\pm}(1,1,1;q_{1},q_{2},q_{1}+q_{2}) = B(q_{1},q_{2}), \notag 
\end{align} 
where $B(x,y)$ is the Euler Beta function, 
$\tilde{I}_{p_{1},q_{1}}^{\pm}, \tilde{I}_{p_{2},q_2}^{\pm}$ and $\tilde{I}_{p_{3},q_{3}}^{\pm}$ are generalized Fresnel integrals defined by \eqref{generalized_Fresnel_integral_def}, 
and double signs $\pm$ are in same order. 
\end{prop}

\begin{proof} 
If $p_{j} = 1$ for $j=1,2,3$, 
since $q_{1}+q_{2} \in \mathbb{C} \setminus \{ -\mathbb{N} \}$, 
by Theorem \ref{th01} (ii), 
\begin{align} 
\tilde{B}^{\pm}(1,1,1;q_{1},q_{2},q_{1}+q_{2}) 
&= \frac{\tilde{I}_{1,q_{1}}^{\pm} \tilde{I}_{1,q_2}^{\pm}}{\tilde{I}_{1,q_{1}+q_{2}}^{\pm}} 
= \frac{e^{\mp i \frac{\pi}{2} q_{1}} \tilde{I}_{1,q_{1}}^{\pm} \cdot e^{\mp i \frac{\pi}{2} q_{2}} \tilde{I}_{1,q_{2}}^{\pm}}
{e^{\mp i \frac{\pi}{2} (q_{1}+q_{2}) } \tilde{I}_{1,q_{1}+q_{2}}^{\pm}} \notag \\ 
&= \frac{\varGamma (q_{1}) \varGamma (q_{2})}{\varGamma (q_{1}+q_{2})} 
= B(q_{1},q_{2}). \notag 
\end{align} 
\end{proof} 

\section{Applications to asymptotic expansions} 

In this section, 
we consider applications of generalized Fresnel integrals to asymptotic expansion, which gives an extension of the stationary phase method in one variable. 

First by \eqref{generalized_Fresnel_integral_def} in \S 4, we define the following generalized Fresnel integrals: 
\begin{defn} 
\label{generalized_Fresnel_pm_pm_m} 
Let $m,k \in \mathbb{N}$. Then we define the generalized Fresnel integrals $\tilde{I}_{m,k}^{\pm \pm_{m}}$ as follows: 
\begin{align} 
\tilde{I}_{m,k}^{\pm \pm_{m}} 
:= Os\mbox{-}\int_{0}^{\infty} e^{\pm (-1)^{m} ix^{m}} x^{k-1} dx 
= m^{-1} e^{\pm (-1)^{m} i\frac{\pi}{2} \frac{k}{m}} \varGamma \left( \frac{k}{m} \right), 
\label{generalized_Fresnel_integral_def_pm_pm_m_type} 
\end{align} 
where double signs $\pm$ are in same order. 
\end{defn} 

Next by Theorem \ref{Lax02} in $\S 3$ and Theorem \ref{th01} in $\S 4$, 
we obtain the following theorem: 
\begin{thm} 
\label{th02} 
Assume that $\lambda > 0$ and $p>0$. 
Let $a \in \mathcal{A}^{\tau}_{\delta}(\mathbb{R})$. 
Then the following hold: 
 \begin{enumerate} 
\item[(i)] 
For any $N \in \mathbb{N}$ such that $N \geq p+1$, 
\begin{align} 
\tilde{I}_{p,1}^{\pm}[a](\lambda) 
:= Os\text{-}\int_{0}^{\infty} e^{\pm i\lambda x^{p}} a(x) dx 
= \sum_{k=0}^{N-[p]-1} \tilde{I}_{p,k+1}^{\pm} \frac{a^{(k)}(0)}{k!} \lambda ^{-\frac{k+1}{p}} + R_{N}^{\pm}(\lambda) 
\notag 
\end{align} 
and 
\begin{align} 
R_{N}^{\pm}(\lambda) 
:= &\sum_{k=N-[p]}^{N-1} \tilde{I}_{p,k+1}^{\pm} \frac{a^{(k)}(0)}{k!} \lambda ^{-\frac{k+1}{p}} 
+ \frac{1}{N!} Os\text{-} \int_{0}^{\infty} e^{\pm i\lambda x^{p}} x^{N} a^{(N)} (\theta x) dx, \notag 
\end{align} 
where $0 < \theta < 1$ and 
$\tilde{I}_{p,k+1}^{\pm}$ are generalized Fresnel integrals defined by \eqref{generalized_Fresnel_integral_def} in \S 4. 
And then there exists a positive constant $C_{N}$ such that for any $\lambda \geq 1$, 
\begin{align} 
| R_{N}^{\pm}(\lambda) | \leq C_{N} (\max_{k < N} |a^{(k)}(0)| + |a|^{(\tau)}_{N+l_{0}+ l_{p,N+1}}) \lambda^{-\frac{N-p+1}{p}}, 
\notag 
\end{align} 
where $l_{0}:=[(N+1)/p)$ and $l_{p,N+1} := [(N+1+\tau)^{+}/(p-1-\delta)]+1$. 

\item[(ii)] 
If $p=m \in \mathbb{N}$, 
then for any $N \in \mathbb{N}$ such that $N > m$, 
\begin{align} 
\tilde{J}_{m}^{\pm}[a](\lambda) 
:= Os\text{-}\int_{-\infty}^{\infty} e^{\pm i\lambda x^{m}} a(x) dx 
= \sum_{k=0}^{N-m-1} \tilde{c}_{k}^{\pm} \frac{a^{(k)}(0)}{k!} \lambda ^{-\frac{k+1}{m}} + \tilde{R}_{N}^{\pm}(\lambda) 
\notag 
\end{align} 
and 
\begin{align} 
\tilde{R}_{N}^{\pm}(\lambda) 
:= &\sum_{k=N-m}^{N-1} \tilde{c}_{k}^{\pm} \frac{a^{(k)}(0)}{k!} \lambda ^{-\frac{k+1}{m}} 
+ \frac{1}{N!} Os\text{-} \int_{-\infty}^{\infty} e^{\pm i\lambda x^{m}} x^{N} a^{(N)} (\theta x) dx, \notag 
\end{align} 
where $0 < \theta < 1$ and 
\begin{align} 
\tilde{c}_{k}^{\pm} 
:&=\tilde{I}_{m,k+1}^{\pm} + (-1)^{k} \tilde{I}_{m,k+1}^{\pm \pm_{m}} \notag \\ 
&= m^{-1} \left\{ e^{\pm i\frac{\pi}{2} \frac{k+1}{m}} + (-1)^{k} e^{\pm (-1)^{m} i\frac{\pi}{2} \frac{k+1}{m}} \right\} \varGamma \left( \frac{k+1}{m} \right). 
\notag 
\end{align} 
And then there exists a positive constant $\tilde{C}_{N}$ such that for any $\lambda \geq 1$, 
\begin{align} 
| \tilde{R}_{N}^{\pm}(\lambda) | \leq \tilde{C}_{N} (\max_{k < N} |a^{(k)}(0)| + |a|^{(\tau)}_{N+l_{0}+ l_{m,N+1}}) \lambda^{-\frac{N-m+1}{m}}, 
\notag 
\end{align} 
where $l_{0}:=[(N+1)/m)$ and $l_{m,N+1} := [(N+1+\tau)^{+}/(m-1-\delta)]+1$, 
\end{enumerate} 
where double signs $\pm$ are in same order. 
\end{thm} 

\begin{proof} 
Since the lower side of double signs $\pm$ can be obtained as the conjugate of the upper one, 
we shall show the upper one. 

(i) 
Suppose $p > 0$. 
By Theorem \ref{Lax02} (iv) in $\S 3$, there exists the following oscillatory integral: 
\begin{align} 
\tilde{I}_{p,1}^{+}[a](\lambda) 
:= Os\text{-}\int_{0}^{\infty} e^{i\lambda x^{p}} a(x) dx. 
\notag 
\end{align} 
By Taylor expansion of $a(x)$ at $x=0$, for any $N \in \mathbb{N}$ such that $N \geq p+1$, 
\begin{align} 
\tilde{I}_{p,1}^{+}[a](\lambda) 
= Os\text{-}\int_{0}^{\infty} e^{i\lambda x^{p}} \left\{ \sum_{k=0}^{N-1} \frac{a^{(k)}(0)}{k!} x^{k} + \frac{x^{N}}{N!} a^{(N)} (\theta x) \right\} dx, 
\notag 
\end{align} 
where $0 < \theta < 1$. 
By Remark \ref{rem01} in $\S 3$, $a^{(N)}(\theta x) \in \mathcal{A}^{\tau+\delta N}_{\delta}(\mathbb{R})$. 
Then by Theorem \ref{Lax02} (iv) in $\S 3$, 
there exist the following oscillatory integrals: 
\begin{align} 
I_{1}(\lambda)
:= \sum_{k=0}^{N-1} \frac{a^{(k)}(0)}{k!} Os\text{-}\int_{0}^{\infty} e^{i\lambda x^{p}} x^{k} dx 
\notag 
\end{align} 
and 
\begin{align} 
I_{2}(\lambda) 
:= \frac{1}{N!} \tilde{I}_{p,N+1}^{+}[a^{(N)} (\theta x)](\lambda) 
:= \frac{1}{N!} Os\text{-}\int_{0}^{\infty} e^{i\lambda x^{p}} x^{N} a^{(N)} (\theta x) dx. 
\notag 
\end{align} 
Hence 
\begin{align} 
\tilde{I}_{p,1}^{+}[a](\lambda) 
= I_{1}(\lambda) + I_{2}(\lambda). 
\label{I_1+I_2} 
\end{align} 

As to $I_{1}(\lambda)$, 
by change of variable $x=\lambda^{-1/p}y$ combining with Theorem \ref{th01} (i) in $\S 4$, we have 
\begin{align} 
I_{1}(\lambda) 
&= \sum_{k=0}^{N-1} \frac{a^{(k)}(0)}{k!} \lim_{\varepsilon \to +0} \lim_{\substack{u \to +0\\v \to \infty}} \int_{u}^{v} e^{i\lambda x^{p}} x^{k} \chi(\varepsilon x) dx \notag \\ 
&= \sum_{k=0}^{N-1} \frac{a^{(k)}(0)}{k!} \lim_{\varepsilon \to +0} \lim_{\substack{u \to +0\\v \to \infty}} \int_{\lambda^{1/p} u}^{\lambda^{1/p} v} e^{iy^{p}} (\lambda^{-\frac{1}{p}} y)^{k} \chi (\varepsilon \lambda^{-\frac{1}{p}} y) \lambda^{-\frac{1}{p}} dy \notag \\ 
&= \sum_{k=0}^{N-1} \tilde{I}_{p,k+1}^{+} \frac{a^{(k)}(0)}{k!} \lambda ^{-\frac{k+1}{p}}, 
\label{I_1} 
\end{align} 
where $\chi \in \mathcal{S}(\mathbb{R})$ with $\chi(0) = 1$, $0 < \varepsilon < 1$ 
and $\tilde{I}_{p,k+1}^{+}$ is a generalized Fresnel integral defined by \eqref{generalized_Fresnel_integral_def} in \S 4. 

As to $I_{2}(\lambda)$, 
since $N+1 > p$, 
by Theorem \ref{Lax02} (v) and \eqref{A_tau_delta_def03} in $\S 3$, 
there exists a positive constant $C_{p,N+1}$ such that for any $\lambda \geq 1$, 
\begin{align} 
| I_{2}(\lambda) | 
\leq C_{p,N+1} (N!)^{-1} |a|^{(\tau)}_{N+l_{0}+ l_{p,N+1}} \lambda^{-\frac{N-p+1}{p}}, 
\label{I_2_lambda_est} 
\end{align} 
where $l_{0}:=[(N+1)/p)$ and $l_{p,N+1} := [(N+1+\tau)^{+}/(p-1-\delta)]+1$. 
Here in order to find same or lower order terms than $\lambda^{-(N-p+1)/p}$ in \eqref{I_1}, 
solving the inequality $\lambda ^{-(k+1)/p} \leq \lambda ^{-(N-p+1)/p}$ on $k \in \mathbb{Z}_{\geq 0}$ for any $\lambda \geq 1$, then $k \geq N-p$. 
Since $[p] \leq p < [p]+1$, then $N-[p] \geq N-p > N-[p]-1$. 
Thus let $k=N-[p]$. Then since $k \geq N-p$, 
\begin{align} 
\bigg| \sum_{k=N-[p]}^{N-1} \tilde{I}_{p,k+1}^{+} \frac{a^{(k)}(0)}{k!} \lambda ^{-\frac{k+1}{p}} \bigg| 
\leq C \max_{k < N} |a^{(k)}(0)| \lambda^{-\frac{N-p+1}{p}}, 
\label{sum_N-[p)-1_N-1_est} 
\end{align} 
where $C = \max_{k=N-[p],\dots,N-1} (| \tilde{I}_{p,k+1}^{+} |/k!)$. 
Therefore put 
\begin{align} 
R_{N}^{+}(\lambda) 
:= &\sum_{k=N-[p]}^{N-1} \tilde{I}_{p,k+1}^{+} \frac{a^{(k)}(0)}{k!} \lambda ^{-\frac{k+1}{p}} + I_{2}(\lambda). 
\notag 
\end{align} 
Then 
according to \eqref{I_1+I_2}, \eqref{I_1}, \eqref{sum_N-[p)-1_N-1_est} and \eqref{I_2_lambda_est}, we obtain 
\begin{align} 
\tilde{I}_{p,1}^{+}[a](\lambda) 
= \sum_{k=0}^{N-[p]-1} \tilde{I}_{p,k+1}^{+} \frac{a^{(k)}(0)}{k!} \lambda ^{-\frac{k+1}{p}} + R_{N}^{+}(\lambda) 
\notag 
\end{align} 
and for any $\lambda \geq 1$, 
\begin{align} 
| R_{N}^{+}(\lambda) | 
\leq C_{N} (\max_{k < N} |a^{(k)}(0)| + |a|^{(\tau)}_{N+l_{0}+ l_{p,N+1}}) \lambda^{-\frac{N-p+1}{p}}, 
\notag 
\end{align} 
where $C_{N} = \max \{ C, C_{p,N+1} (N!)^{-1} \}$. 

(ii) 
Suppose $m \in \mathbb{N}$. 
By Theorem \ref{Os_m} (iii) in $\S 3$, there exists the following oscillatory integral: 
\begin{align} 
\tilde{J}_{m}^{+}[a](\lambda) 
:= Os\text{-}\int_{-\infty}^{\infty} e^{i\lambda x^{m}} a(x) dx. 
\notag 
\end{align} 

We take $N \in \mathbb{N}$ such that $N > m$. Since $N \geq m+1$ and $[m]=m$, by (i), we have 
\begin{align} 
\tilde{I}_{m,1}^{+}[a](\lambda) 
:= Os\text{-}\int_{0}^{\infty} e^{i\lambda x^{m}} a(x) dx 
&= \sum_{k=0}^{N-m-1} \tilde{I}_{m,k+1}^{+} \frac{a^{(k)}(0)}{k!} \lambda ^{-\frac{k+1}{m}} + R_{N}^{+}(\lambda) 
\label{tilde_I_m_1_+_a} 
\end{align} 
and 
\begin{align} 
R_{N}^{+}(\lambda) 
:= &\sum_{k=N-m}^{N-1} \tilde{I}_{m,k+1}^{+} \frac{a^{(k)}(0)}{k!} \lambda ^{-\frac{k+1}{m}} 
+ \frac{1}{N!} Os\text{-} \int_{0}^{\infty} e^{i\lambda x^{m}} x^{N} a^{(N)} (\theta x) dx, 
\label{R_N_+} 
\end{align} 
where $0 < \theta < 1$ and $\tilde{I}_{m,k+1}^{+}$ is a generalized Fresnel integral defined by \eqref{generalized_Fresnel_integral_def} in \S 4.  
And then by (i), there exists a positive constant $C_{m,N}^{(1)}$ such that for any $\lambda \geq 1$, 
\begin{align} 
| R_{N}^{+}(\lambda) | 
\leq C_{m,N}^{(1)} (\max_{k < N} |a^{(k)}(0)| + |a|^{(\tau)}_{N+l_{0}+ l_{m,N+1}}) \lambda^{-\frac{N-m+1}{m}}, 
\label{R_N_+_lambda_est_C_3} 
\end{align} 
where $l_{0}:=[(N+1)/m)$ and $l_{m,N+1} := [(N+1+\tau)^{+}/(m-1-\delta)]+1$. 

By \eqref{Os_m_varphi_change_of_variable} and \eqref{Os_m_psi_change_of_variable} in $\S 3$, since $\varphi + \psi \equiv 1$, 
\begin{align} 
&Os\text{-}\int_{-\infty}^{0} e^{i\lambda x^{m}} a(x) dx 
= Os\text{-}\int_{0}^{\infty} e^{(-1)^{m}i\lambda y^{m}} a(-y) dy. 
\label{Os_int_-infty_0_e_i_lambda_x_m_a(x)_dx} 
\end{align} 

Hence thanks to \eqref{tilde_I_m_1_+_a}, \eqref{R_N_+} and \eqref{Os_int_-infty_0_e_i_lambda_x_m_a(x)_dx}, we obtain 
\begin{align} 
Os\text{-}\int_{-\infty}^{0} e^{i\lambda x^{m}} a(x) dx 
= \sum_{k=0}^{N-m-1} (-1)^{k} \tilde{I}_{m,k+1}^{\pm_{m}} \frac{a^{(k)}(0)}{k!} \lambda ^{-\frac{k+1}{m}} + R_{N}^{\pm_{m}}(\lambda) 
\label{Os_int_-infty_0} 
\end{align} 
and 
\begin{align} 
&R_{N}^{\pm_{m}}(\lambda) \notag \\ 
:&= \sum_{k=N-m}^{N-1} (-1)^{k} \tilde{I}_{m,k+1}^{\pm_{m}} \frac{a^{(k)}(0)}{k!} \lambda ^{-\frac{k+1}{m}} + \frac{1}{N!} Os\text{-}\int_{0}^{\infty} e^{(-1)^{m}i\lambda y^{m}} y^{N} a^{(N)}(-\theta y) dy \notag \\ 
&= \sum_{k=N-m}^{N-1} (-1)^{k} \tilde{I}_{m,k+1}^{\pm_{m}} \frac{a^{(k)}(0)}{k!} \lambda ^{-\frac{k+1}{m}} + \frac{1}{N!} Os\text{-} \int_{-\infty}^{0} e^{i\lambda x^{m}} x^{N} a^{(N)} (\theta x) dx, 
\label{R_N_pm_m} 
\end{align} 
where 
$\tilde{I}_{m,k+1}^{\pm_{m}}$ is a generalized Fresnel integral defined by \eqref{generalized_Fresnel_integral_def_pm_pm_m_type}. 
And then by (i), similarly to \eqref{R_N_+_lambda_est_C_3}, there exists a positive constant $C_{m,N}^{(2)}$ such that for any $\lambda \geq 1$, 
\begin{align} 
| R_{N}^{\pm_{m}}(\lambda) | 
\leq C_{m,N}^{(2)} (\max_{k < N} |a^{(k)}(0)| + |a|^{(\tau)}_{N+l_{0}+ l_{m,N+1}}) \lambda^{-\frac{N-m+1}{m}}, 
\label{R_N_pm_lambda_est_C_4} 
\end{align} 
where $l_{0}:=[(N+1)/m)$ and $l_{m,N+1} := [(N+1+\tau)^{+}/(m-1-\delta)]+1$. 

Hence plugging \eqref{R_N_pm_m} into \eqref{R_N_+}, put 
\begin{align} 
\tilde{R}_{N}^{+}(\lambda) 
:&= \sum_{k=N-m}^{N-1} \tilde{c}_{k}^{+} \frac{a^{(k)}(0)}{k!} \lambda ^{-\frac{k+1}{m}} 
+ \frac{1}{N!} Os\text{-} \int_{-\infty}^{\infty} e^{i\lambda x^{m}} x^{N} a^{(N)} (\theta x) dx, 
\notag 
\end{align} 
where 
\begin{align} 
\tilde{c}_{k}^{+} 
&:= \tilde{I}_{m,k+1}^{+} + (-1)^{k} \tilde{I}_{m,k+1}^{\pm_{m}}. 
\notag 
\end{align} 
Then \eqref{tilde_I_m_1_+_a} and \eqref{Os_int_-infty_0} lead us into 
\begin{align} 
\tilde{J}_{m}^{+}[a](\lambda) 
:= Os\text{-}\int_{-\infty}^{\infty} e^{i\lambda x^{m}} a(x) dx 
&= \sum_{k=0}^{N-m-1} \tilde{c}_{k}^{+} \frac{a^{(k)}(0)}{k!} \lambda ^{-\frac{k+1}{m}} + \tilde{R}_{N}^{+}(\lambda), 
\notag 
\end{align} 
and by \eqref{R_N_+_lambda_est_C_3} and \eqref{R_N_pm_lambda_est_C_4}, for any $\lambda \geq 1$, we have 
\begin{align} 
| \tilde{R}_{N}^{+}(\lambda) | 
\leq \tilde{C}_{N} (\max_{k < N} |a^{(k)}(0)| + |a|^{(\tau)}_{N+l_{0}+ l_{m,N+1}}) \lambda^{-\frac{N-m+1}{m}}, 
\notag 
\end{align} 
where $\tilde{C}_{N} = \max \{ C_{m,N}^{(1)}, C_{m,N}^{(2)} \}$. 
This complete the proof. 
\end{proof} 

In particular, if $m=2l-1$ and $m=2l$ for $l \in \mathbb{N}$, then the following holds: 
\begin{cor} 
\label{cor01} 
Assume that $\lambda > 0$. 
For any $m \in \mathbb{N}$, 
the following hold: 
\begin{enumerate} 
\item[(i)] 
If $m=2l-1$ for $l \in \mathbb{N}$, then for any $N \in \mathbb{N}$ such that $N \geq l$, as $\lambda \to \infty$, 
\begin{align} 
&\tilde{J}_{2l-1}^{\pm}[a](\lambda) 
:= Os\text{-}\int_{-\infty}^{\infty} e^{\pm i\lambda x^{2l-1}} a(x) dx \notag \\ 
&= \frac{2}{2l-1} \sum_{k=0}^{N-l} \left\{ \cos \frac{\pi (2k+1)}{2(2l-1)} \varGamma \left( \frac{2k+1}{2l-1} \right) \frac{a^{(2k)}(0)}{(2k)!} \lambda ^{-\frac{2k+1}{2l-1}} \right. \notag \\ 
&\hspace{0.5cm}\left. \pm i \sin \frac{\pi (2k+2)}{2(2l-1)} \varGamma \left( \frac{2k+2}{2l-1} \right) \frac{a^{(2k+1)}(0)}{(2k+1)!} \lambda ^{-\frac{2k+2}{2l-1}} \right\} + O\left( \lambda ^{-\frac{2(N-l+1)}{2l-1}} \right). 
\notag 
\end{align} 
\item[(ii)] 
If $m=2l$ for $l \in \mathbb{N}$, then for any $N \in \mathbb{N}$ such that $N \geq l$, as $\lambda \to \infty$, 
\begin{align} 
\tilde{J}_{2l}^{\pm}[a](\lambda) 
:&= Os\text{-}\int_{-\infty}^{\infty} e^{\pm i\lambda x^{2l}} a(x) dx \notag \\ 
&= 2 \sum_{k=0}^{N-l} \tilde{I}_{2l,2k+1}^{\pm} \frac{a^{(2k)}(0)}{(2k)!} \lambda ^{-\frac{2k+1}{2l}} + O\left( \lambda ^{-\frac{N-l+1}{l}} \right), 
\notag 
\end{align} 
where 
\begin{align} 
\tilde{I}_{2l,2k+1}^{\pm} 
:= Os\mbox{-}\int_{0}^{\infty} e^{\pm ix^{2l}} x^{2k} dx 
= (2l)^{-1} e^{\pm i\frac{\pi}{2} \frac{2k+1}{2l}} \varGamma \left( \frac{2k+1}{2l} \right) 
\notag 
\end{align} 
are generalized Fresnel integrals defined by \eqref{generalized_Fresnel_integral_def} in \S 4. 
\item[(iii)] 
If $m=1$, then for any $N \in \mathbb{N}$, as $\lambda \to \infty$, 
\begin{align} 
\tilde{J}_{1}^{\pm}[a](\lambda) 
:&= Os\text{-}\int_{-\infty}^{\infty} e^{\pm i\lambda x} a(x) dx 
= O\left( \lambda ^{-2N} \right). 
\notag 
\end{align} 
\item[(iv)] 
If $m=2$, then for any $N \in \mathbb{N}$, as $\lambda \to \infty$, 
\begin{align} 
\tilde{J}_{2}^{\pm}[a](\lambda) 
:&= Os\text{-}\int_{-\infty}^{\infty} e^{\pm i\lambda x^{2}} a(x) dx \notag \\ 
&= \sqrt{\pi} \sum_{k=0}^{N-1} 
e^{\pm i\frac{\pi}{2}(k+\frac{1}{2})} \frac{a^{(2k)}(0)}{4^{k} k!} \lambda ^{-k-\frac{1}{2}} + O\left( \lambda ^{-N} \right), 
\notag 
\end{align} 
\end{enumerate} 
where double signs $\pm$ are in same order. 
\end{cor} 

\begin{proof} 
(i) Let $m = 2l-1$ and $N \in \mathbb{N}$ such that $N \geq l$. Then since $2N \geq 2l > 2l-1 = m$, by Theorem \ref{th02} (ii), as $\lambda \to \infty$, 
\begin{align} 
\tilde{J}_{2l-1}^{\pm}[a](\lambda) 
= \sum_{j=0}^{2N-(2l-1)-1} \tilde{c}_{j}^{\pm} \frac{a^{(j)}(0)}{j!} \lambda ^{-\frac{j+1}{2l-1}} + O\left( \lambda ^{-\frac{2N-2l+2}{2l-1}} \right), 
\notag 
\end{align} 
where 
\begin{align} 
\tilde{c}_{j}^{\pm} 
&= (2l-1)^{-1} \left\{ e^{\pm i\frac{\pi}{2} \frac{j+1}{2l-1}} + (-1)^{j} e^{\mp i\frac{\pi}{2} \frac{j+1}{2l-1}} \right\} \varGamma \left( \frac{j+1}{2l-1} \right) \notag \\ 
&= 
\begin{cases} 
\dfrac{2}{2l-1} \cos \dfrac{\pi(2k+1)}{2(2l-1)} \varGamma \left( \dfrac{2k+1}{2l-1} \right) & \text{for $j=2k$}, \\ 
\dfrac{\pm 2i}{2l-1} \sin \dfrac{\pi(2k+2)}{2(2l-1)} \varGamma \left( \dfrac{2k+2}{2l-1} \right) & \text{for $j=2k+1$}. 
\end{cases} 
\notag 
\end{align} 
Here we used $e^{\pm i\theta} + e^{\mp i\theta} = 2\cos \theta$ and $e^{\pm i\theta} - e^{\mp i\theta} = \pm 2i\sin \theta$ for $\theta \in \mathbb{R}$. 
Hence as $\lambda \to \infty$, 
\begin{align} 
&\tilde{J}_{2l-1}^{\pm}[a](\lambda) 
= \frac{2}{2l-1} \sum_{k=0}^{N-l} \left\{ \cos \frac{\pi(2k+1)}{2(2l-1)} \varGamma \left( \frac{2k+1}{2l-1} \right) \frac{a^{(2k)}(0)}{(2k)!} \lambda ^{-\frac{2k+1}{2l-1}} \right. \notag \\ 
&\hspace{0.25cm}\left. \pm i \sin \frac{\pi(2k+2)}{2(2l-1)} \varGamma \left( \frac{2k+2}{2l-1} \right) \frac{a^{(2k+1)}(0)}{(2k+1)!} \lambda ^{-\frac{2k+2}{2l-1}} \right\} + O\left( \lambda ^{-\frac{2(N-l+1)}{2l-1}} \right). 
\notag 
\end{align} 

(ii) Let $m = 2l$ and $N \in \mathbb{N}$ such that $N \geq l$. Then since $2N+1 > 2N \geq 2l = m$, by Theorem \ref{th02} (ii), as $\lambda \to \infty$, 
\begin{align} 
\tilde{J}_{2l}^{\pm}[a](\lambda) 
= \sum_{j=0}^{2N+1-2l-1} \tilde{c}_{j}^{\pm} \frac{a^{(j)}(0)}{j!} \lambda ^{-\frac{j+1}{2l}} + O\left( \lambda ^{-\frac{2N+1-2l+1}{2l}} \right), 
\notag 
\end{align} 
where 
\begin{align} 
\tilde{c}_{j}^{\pm} 
:= \tilde{I}_{2l,j+1}^{\pm} + (-1)^{j} \tilde{I}_{2l,j+1}^{\pm} 
&= 
\begin{cases} 
2 \tilde{I}_{2l,2k+1}^{\pm} & \text{for $j=2k$}, \\ 
0 & \text{for $j=2k+1$} 
\end{cases} 
\notag 
\end{align} 
for $k \in \mathbb{Z}_{\geq 0}$. 
Hence as $\lambda \to \infty$, 
\begin{align} 
\tilde{J}_{2l}^{\pm}[a](\lambda) 
= 2 \sum_{k=0}^{N-l} \tilde{I}_{2l,2k+1}^{\pm} \frac{a^{(2k)}(0)}{(2k)!} \lambda ^{-\frac{2k+1}{2l}} + O\left( \lambda ^{-\frac{N-l+1}{l}} \right). 
\notag 
\end{align} 

(iii) Let $l=1$ and $N \in \mathbb{N}$. Then since $N \geq 1 = l$, by (i), as $\lambda \to \infty$, 
\begin{align} 
\tilde{J}_{1}^{\pm}[a](\lambda) 
&= 2 \sum_{k=0}^{N-1} \left\{ \cos \frac{\pi(2k+1)}{2} \varGamma (2k+1) \frac{a^{(2k)}(0)}{(2k)!} \lambda ^{-2k-1} \right. \notag \\ 
&\hspace{0.75cm}\left. \pm i \sin \frac{\pi(2k+2)}{2} \varGamma (2k+2) \frac{a^{(2k+1)}(0)}{(2k+1)!} \lambda ^{-2k-2} \right\} + O\left( \lambda ^{-2N} \right) \notag \\ 
&= O\left( \lambda ^{-2N} \right). 
\notag 
\end{align} 

(iv) Let $l=1$ and $N \in \mathbb{N}$. Then since $N \geq 1 = l$, by (ii), as $\lambda \to \infty$, 
\begin{align} 
\tilde{J}_{2}^{\pm}[a](\lambda) 
&= \sum_{k=0}^{N-1} 
e^{\pm i\frac{\pi}{2} \frac{2k+1}{2}} \varGamma \left( \frac{2k+1}{2} \right) \frac{a^{(2k)}(0)}{(2k)!} \lambda ^{-\frac{2k+1}{2}} + O\left( \lambda ^{-N} \right) \notag \\ 
&= \sum_{k=0}^{N-1} 
e^{\pm i\frac{\pi}{2}(k+\frac{1}{2})} \frac{(2k-1)!!}{2^{k}} \sqrt{\pi} \frac{a^{(2k)}(0)}{(2k)!! (2k-1)!!} \lambda ^{-k-\frac{1}{2}} + O\left( \lambda ^{-N} \right) \notag \\ 
&= \sqrt{\pi} \sum_{k=0}^{N-1} 
e^{\pm i\frac{\pi}{2}(k+\frac{1}{2})} \frac{a^{(2k)}(0)}{4^{k} k!} \lambda ^{-k-\frac{1}{2}} + O\left( \lambda ^{-N} \right). 
\notag 
\end{align} 
\end{proof} 

Corollary \ref{cor01} (iii) gives the extension of the property that Fourier transform of $a(x)$ is of $O(\lambda ^{-2N})~(\lambda \to \infty)$ for any $N \in \mathbb{N}$ in $\mathcal{S}(\mathbb{R})$ to $\mathcal{A}^{\tau}_{\delta}(\mathbb{R})$. 
Also by Example \ref{stationary example} in \S 2 and Corollary \ref{cor01} (iv), 
we can consider that Theorem \ref{th02} (ii) is a generalization of Proposition \ref{quadratic phase} in \S 2, 
which is a principal part of the stationary phase method in one variable. 

To the end of the present paper, 
we note that we obtain asymptotic expansion of oscillatory integral in several variables with not only the Morse type phase functions but also singular types 
for examples $A_{k}$, $E_6$, $E_8$-phase functions. 
For details, see \cite{Nagano-Miyazaki02}.


\begin{thebibliography}{99}

\bibitem{AGV}
\textsc{V.I.~Arnold, S.M.~Gusein-Zade and A.N.~Var\v{c}enko},
\textit{Singularities of differentiable maps volume I,II},
Birkh\"{a}user,
I, 1985, II, 1988. 

\bibitem{Arnold01}
\textsc{V.I.~Arnold},
\textit{Normal forms of functions with simple critical points, the Weyl groups $A_{k}$, $D_{k}$, $E_{k}$, and Lagrange immersions},
Functional Anal. Appl. 254-272, 
(Russian original, Funkc. anal. i prilo\v{z}. \textbf{6}-4 (1972),3--25).

\bibitem{Duistermaat01}
\textsc{J.J.~Duistermaat},
\textit{Fourier integral operators},
Birkh\"{a}user, Boston. 
1996. 

\bibitem{Duistermaat02}
\textsc{J.J.~Duistermaat},
\textit{Oscillatory integrals, Lagrange immersions and unfoldings of singularities},
Comm. Pure Appl. Math. 
\textbf{27} (1974), 
207--281. 

\bibitem{Fujiwara1}
\textsc{D.~Fujiwara},
\textit{Asymptotic method in linear partial differential equations},
Iwanami Shoten,
2019 (in Japanese). 

\bibitem{Fujiwara2} 
\textsc{D.~Fujiwara}, 
\textit{Mathematical method for Feynman path integrals},
Springer, Tokyo,
1999 (in Japanese).

\bibitem{Fujiwara3} 
\textsc{D.~Fujiwara}, 
\textit{Rigorous time slicing approach to Feynman path integrals},
Springer,
2017.

\bibitem{Gel'fand-Shilov}
\textsc{I.M.~Gel'fand and G.E.~Shilov},
\textit{Generalized functions, volume 1 : properties and operations},
AMS Chelsea Publishing 
2016 (Russian original, Fizmatgiz, 1958).

\bibitem{Grigis-Sjostrand}
\textsc{A.~Grigis and J.~Sj\"{o}strand},
\textit{Microlocal analysis for differential operators, an introduction},
London Mathematical Society Lecture Note Series \textbf{196}, 
Cambridge University Press,
Cambridge, 
1994.

\bibitem{Hormander01}
\textsc{L.~H\"{o}rmander},
\textit{Fourier integral operators I},
Acta Math. 
\textbf{127} (1971), 
79--183. 

\bibitem{Hormander02}
\textsc{L.~H\"{o}rmander},
\textit{The analysis of linear partial differential operators I},
Springer,
1983. 

\bibitem{Ito-Komatsu}
\textsc{S.~Ito and H.~Komatsu edit},
\textit{Foundation of analysis},
Iwanami Shoten,
1977 (in Japanese).

\bibitem{Izumiya-Ishikawa}
\textsc{S.~Izumiya and G.~Ishikawa},
\textit{Theory of applied singular points},
Kyouritu Syuppan,
1998 (in Japanese). 

\bibitem{Kamimoto-Nose} 
\textsc{J. Kamimoto and T. Nose}, 
\textit{Toric resolution of singularities in a certain class of $C^{\infty}$ functions and asymptotic analysis of oscillatory integrals},
J. Math. Sci. Univ. Tokyo 
\textbf{23} (2016), 
425--485. 

\bibitem{Kumano-go}
\textsc{H.~Kumano-go},
\textit{Pseudo-differential operators},
MIT Press,
1981.

\bibitem{Kumano-go02}
\textsc{H.~Kumano-go},
\textit{Theory of Fourier integral operators and pseudo-differential operators and fundamental solution for hyperbolic equations},
Department of Mathematics, Faculty of Science, Osaka University,
1983 (in Japanese). 

\bibitem{Milnor}
\textsc{J.~Milnor},
\textit{Morse theory, Ann. Math. Studies 51},
Princeton Univ. Press,
1963. 

\bibitem{Nagano01}
\textsc{T.~Nagano},
\textit{An extension of phase functions on the singular points in the stationary phase method}, 
Master's Thesis in Tokyo University of Science, 
1994 (in Japanese). 

\bibitem{Nagano-Miyazaki03}
\textsc{T.~Nagano and N.~Miyazaki},
\textit{On singular points and oscillatory integrals}, 
arXiv:1906.01438v1 [math.CA], 
2 Jun 2019. 

\bibitem{Nagano-Miyazaki02}
\textsc{T.~Nagano and N.~Miyazaki},
\textit{On oscillatory integrals associated to phase functions with degenerate singular points}, 
arXiv:2009.09620v1 [math.CA], 
21 Sep 2020. 

\bibitem{Omori01} 
\textsc{H. Omori}, 
\textit{Algebra of calculus from the operational viewpoint I,II},
Gendai Suugaku Sya,
I, 2018, II, 2019 (in Japanese).

\bibitem{OMYK03} 
\textsc{H. Omori,~Y. Maeda,~A. Yoshioka,~O. Kobayashi}, 
\textit{On regular Fr\'{e}chet-Lie groups III, a second cohomology class related to the Lie algebra of pseudo-differential operators of order one},
Tokyo. J. Math. 
\textbf{4} (1981), 
255--277. 

\bibitem{Sugiura}
\textsc{M.~Sugiura},
\textit{Introduction to analysis~I,II},
Univ. of Tokyo Press,
I, 1980, II, 1985 (in Japanese).

\bibitem{YFI}
\textsc{E.~Yoshinaga,~T.~Fukui and S.~Izumi}, 
\textit{Analytic functions and singular points}, 
edited by T. Fukuda, S. Izumiya and G. Ishikawa, 
Kyouritu Syuppan,
2002 (in Japanese). 

\end{thebibliography}
\end{document}